\documentclass{amsart}
\usepackage{latexsym, graphicx, amssymb, cite, hyperref, tikz-cd}
\usepackage[normalem]{ulem}

\newtheorem{theorem}{Theorem}[section] 
\newtheorem{lemma}[theorem]{Lemma}
\newtheorem{proposition}[theorem]{Proposition}

\newtheorem{question}[theorem]{Question}

\theoremstyle{definition}
\newtheorem{definition}[theorem]{Definition}
\newtheorem{example}[theorem]{Example}

\theoremstyle{remark}
\newtheorem{remark}[theorem]{Remark}

\numberwithin{equation}{section}

\newcommand{\R}{{\mathbb R}}

\newcommand{\C}{{\mathbb C}}
\newcommand{\N}{{\mathbb N}}

\let\Re\relax
\DeclareMathOperator{\Re}{Re}

\DeclareMathOperator*{\supp}{supp}

\DeclareMathOperator{\WF}{WF}

\DeclareMathOperator{\vol}{vol}

\DeclareMathOperator{\inj}{inj}

\DeclareMathOperator*{\E}{\mathbb E}

\title{Can you hear your location on a manifold?}

\author{Emmett L. Wyman}
\address{Department of Mathematics, University of Rochester, Rochester NY}
\email{emmett.wyman@rochester.edu}

\author{Yakun Xi}
\address{School of Mathematical Sciences, Zhejiang University, Hangzhou 310027, PR China}
\email{yakunxi@zju.edu.cn}

\dedicatory{In memory of Steve Zelditch}

\begin{document}
	
	\begin{abstract}
		We introduce a variation on Kac's question, ``Can one hear the shape of a drum?" Instead of trying to identify a compact manifold and its metric via its Laplace--Beltrami spectrum, we ask if it is possible to uniquely identify a point $x$ on the manifold, up to symmetry, from its pointwise counting function
		\[
		N_x(\lambda) = \sum_{\lambda_j \leq \lambda} |e_j(x)|^2,
		\]
		where here $\Delta_g e_j = -\lambda_j^2 e_j$ and $e_j$ form an orthonormal basis for $L^2(M)$. 
		
		{ This problem has several natural physical interpretations, two of which are acoustic: 1. You are placed at an arbitrary location in a familiar room with your eyes closed. Can you identify your location in the room by clapping your hands once and listening to the resulting echoes and reverberations?
			2. If a drum of a known shape is struck at some unknown point, can you determine this point by listening to the quality of the sound the drum produces?}
		
		The main result of this paper provides an affirmative answer to this question for a generic class of metrics. We also probe the problem with a variety of simple examples, highlighting along the way helpful geometric invariants that can be pulled out of the pointwise counting function $N_x$. 
	\end{abstract}
	
	\maketitle
	
	\section{Introduction}
	
	Let $(M,g)$ be a compact {connected} Riemannian manifold with or without a boundary. We consider an orthonormal basis of Laplace--Beltrami eigenfunctions $e_j$ for $j = 1,2,\ldots$ satisfying
	\[
	\Delta_g e_j = -\lambda_j^2 e_j,
	\]
	and Dirichlet (or Neumann) boundary condition if $\partial M\neq \emptyset$.
	The Weyl counting function
	\[
	N(\lambda) = \#\{j : \lambda_j \leq \lambda\}
	\]
	counts the number of Laplace--Beltrami eigenvalues (with multiplicity) up to some threshold $\lambda$. It is possible to deduce  a significant amount of geometric information about $M$ from the counting function. The main term of the Weyl law
	\[
	N(\lambda) = (2\pi)^{-n} (\vol M) (\vol B^n) \lambda^n + O(\lambda^{n - 1})
	\]
	identifies both the dimension $n$ of $M$ and its volume. The asymptotics of the heat trace,
	\[
	\int_{-\infty}^\infty e^{-t\lambda^2} \, dN(\lambda)
	\]
	reveals these quantities too, but also the measure of the boundary (if there is one) and the Euler characteristic. The singularities of the wave trace,
	\[
	\int_{-\infty}^\infty e^{-it\lambda} \, dN(\lambda),
	\]
	mark the closed geodesics' lengths, order, and Maslov indices. 
	
	One might ask if these geometric quantities are enough to distinguish one manifold from another via their spectra.  Indeed, this is the essence of Kac's famous question, ``{\it Can one hear the shape of a drum?}" \cite{Kac66}. The answer to this question is complicated and depends very much on the precise setting of the problem.
	
	On the one hand, we now know of a number of examples of pairs of non-isometric manifolds that are isospectral. Milnor \cite{milnor} showed that there exists a pair of 16-dimensional flat tori, which have the same Laplace--Beltrami eigenvalues but different shapes. In 1992, Gordon, Webb, and Wolpert \cite{GWW92} constructed a pair of isospectral concave polygons in the plane with different shapes. In fact, this result settled Kac's original conjecture, which was formulated in terms of planar domains. Their method was later generalized by Buser, Conway, Doyle, and Semmler \cite{BCDS94} to construct numerous similar examples.
	
	On the other hand, there are some situations where one can distinguish a manifold from others in some restricted class. For instance, by the isoperimetric inequality, one can quickly identify a disk via its Dirichlet (or Neumann) spectrum amongst other planar regions with smooth boundaries. Nonetheless, it is challenging to prove any positive result for smooth domains other than the disk. There are several partial positive results \cite{DKW16,HZ12,PT03,PT12,PT16,Vi21,Z04,Z09}, which often require additional assumptions on the class of smooth domains.  Very recently, by applying a local version of the
	Birkhoff conjecture \cite{ADK16,KS18},  Hezari and Zelditch \cite{HZ22} proved that an ellipse of small eccentricity is  spectrally unique among all domains with smooth boundaries.
	
	We consider a variation on the ``can one hear the shape of a drum" problem.
	Suppose we know everything there is to know about the manifold $M$, but now fix some unknown point $x$ in the interior of $M$. At our disposal is complete information about the pointwise Weyl counting function
	\begin{equation}\label{pointwise counting}
		N_x(\lambda) = \sum_{\lambda_j \leq \lambda} |e_j(x)|^2.  
	\end{equation}
	Can we deduce the position of $x$? As before, various transforms of the derivative $dN_x(\lambda)$ reveal geometric quantities related to $M$ and the location of $x$ within it. Taking the cosine transform, in particular, reveals a natural physical interpretation of the problem. Observe:
	\begin{equation}\label{cosine transform exposition}
		\int_\R \cos(t\lambda) \, dN_x(\lambda) = \sum_{j} \cos(t\lambda_j) |e_j(x)|^2 =  \cos(t
		{\sqrt{-\Delta_g}})\delta_x(x).
	\end{equation}
	Here, $\cos(t
	{\sqrt{-\Delta_g}})f$ is the solution operator to the wave equation
	\[
	(\partial_t^2 - \Delta_g)u = 0 \qquad \text{ with } \qquad \begin{cases}
		u(0) = f \\
		\dot u(0) = 0.
	\end{cases}
	\]
	
	Noting that \eqref{cosine transform exposition} and \eqref{pointwise counting} determine one another, we can use the rightmost side of \eqref{cosine transform exposition} to interpret our problem as such: You stand on a manifold, make a single sharp ``snap'' sound, and then listen intently to its reverberations. If you have perfect hearing and perfect knowledge of the shape of the manifold, can you deduce your location within it?
	
	A symmetric manifold certainly defeats us, so instead, we propose:
	
	\begin{question}\label{echolocation question}
		Let $M$ be a compact {connected} Riemannian manifold with or without boundary and fix points $x$ and $y$ in the interior of $M$ for which $N_x = N_y$ identically. Must there be an isometry on $M$ that maps $x$ to $y$?
	\end{question}
	
	\begin{remark} \label{disk example} To illustrate how this kind of ``echolocation'' is plausible, consider a point $x$ in the interior of the unit disk $D$ in $\R^2$ whose Laplace operator $\Delta$ comes equipped with Dirichlet boundary conditions. By the arguments in \cite{Ivrii}, the lengths of billiard trajectories which depart and arrive at $x$ have lengths at singularities $t$ of the distribution
		\[
		\int_{-\infty}^\infty e^{it\lambda} \, dN_x(\lambda) = \sum_j e^{it\lambda_j} |e_j(x)|^2 = e^{it\sqrt{-\Delta_g}}(x,x).
		\]
		The first positive time $t$ at which $e^{it\sqrt{-\Delta}}(x,x)$ is singular is $2d(x,\partial D)$, the length of the shortest such trajectory. The distance $d(x,\partial D)$ determines $x$ uniquely up to rotation about the center of $D$. See Figure \ref{fig1} for an illustration of ``echolocation'' in a general planar domain. 
		
		\begin{figure}[h]
			\centering
			\includegraphics[width=0.8\textwidth]{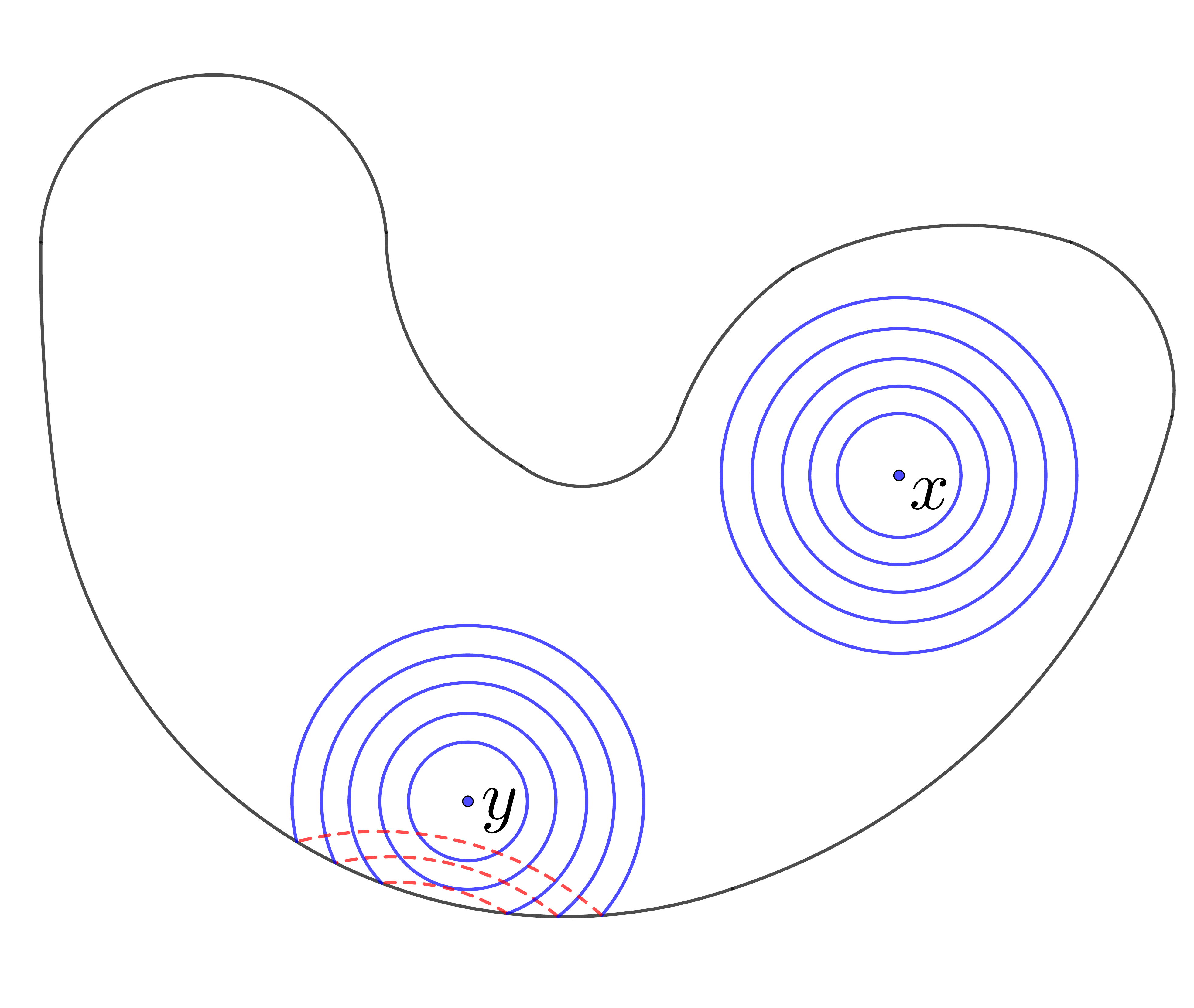}
			\caption{``Echolocation'' in a planar domain.}
			\label{fig1}
		\end{figure}
	\end{remark}
	
	{ Besides the echolocation interpretation, there are a few more physical interpretations of Question \ref{echolocation question}. Among them, the three most interesting ones can be phrased as:
		\begin{itemize}
			\item Can you hear at which point a drum is struck?
			\item Can you locate yourself using Brownian motions of free particles?
			\item Can you find the position of a strictly confined quantum particle, given the state of the quantum superposition after releasing it? 
		\end{itemize}
		In particular, the first interpretation above is closely related to a version of Kac's question introduced in \cite{Conway}. Detailed discussions 
		on different interpretations will be given in Section \ref{sec: interpretations}.}
	
	Our main result shows that the answer to Question \ref{echolocation question} is affirmative for generic manifolds without boundaries. Recall, a subset of a Baire topological space (e.g. complete metric space) is said to be \emph{residual} if it is the complement of a meager set, or equivalently the countable intersection of open dense subsets.
	
	\begin{theorem}\label{thm: generic echolocation}
		Let $M$ be a compact smooth manifold without boundary, $\dim M \geq 2$. Then, there exists a residual class of metrics in the $C^\infty$ topology such that if $N_x = N_y$ for some $x,y \in M$, then $x = y$.
	\end{theorem}
	
	\begin{remark}\label{incomplete remark}
		It is not unusual to prove generic statements regarding eigenfunctions and eigenvalues of the Laplacian on a manifold. For instance, Uhlenbeck \cite{Uhlenbeck} showed that, generically, the eigenspaces are all simple. Indeed, this result of Uhlenbeck further connects our Question \ref{echolocation question} to the original hearing the shape of a drum problem. For a manifold with simple eigenspaces, one should be able to hear all the eigenvalues of the Laplacian from $N_x$, as long as $e_j(x)\neq 0$ for all $j\ge1$, and hence recover the Weyl counting function $N(\lambda)$ from $N_x$ for a generic choice of $x$.
	\end{remark}

	\subsection{Structure of the paper}
	{In Section \ref{sec: interpretations}, we discuss other natural interpretations of Question \ref{echolocation question}.}
	Section \ref{sec: examples} houses a number of simple examples of manifolds on which echolocation is possible. In some of these examples, we use the pointwise counting function $N_x$ directly. In others, we pull geometric invariants out of the pointwise counting function and use those.
	
	Sections \ref{sec: generic injections}, \ref{sec: proof of main}, and \ref{sec: proof lemma} are dedicated to the proof of our main result, Theorem \ref{thm: generic echolocation}. In Section \ref{sec: generic injections}, we develop a somewhat general criterion for showing injections are generic in the topological sense. In Section \ref{sec: proof of main}, we use this criterion to reduce Theorem \ref{thm: generic echolocation} to the construction of a special set of perturbations of the Riemannian metric. Section \ref{sec: proof lemma} contains the key stationary phase argument needed to complete the construction.
	
	Finally, we discuss further directions of research in Section \ref{sec: further directions}.

	\subsection*{Acknowledgements}Xi was  supported by the National Key Research and Development
	Program of China No. 2022YFA1007200
	and NSF China Grant No. 12171424. Wyman was partially supported by NSF grant DMS-2204397 and by the AMS-Simons Travel Grants. The authors would like to thank Allan Greenleaf, Hadmid Hezari, Alex Iosevich, Steven Kleene, Jonathan Pakianathan, Chris Sogge, Chengbo Wang, and Meng Wang for helpful conversations.
	
	{\section{Further physical interpretations}
		\label{sec: interpretations}
		{In this section, we present a few more physical inverse problems mathematically equivalent to Question \ref{echolocation question}.}
		\subsection{Can you hear where a drum is struck?}
		
		We offer another acoustic 
		interpretation of Question \ref{echolocation question}.  This interpretation has a closer relationship to Kac's \cite{Kac66} original question: ``Can one hear the shape of a drum?''
		
		In the work of Buser, Conway, Doyle and Semmler \cite{Conway}, they identify a pair of domains which are not only isospectral, but also {\it homophonic}. They define two planar domains to be {homophonic} if each domain has a distinguished point such that corresponding normalized Dirichlet eigenfunctions yield  equal values at the distinguished
		points. 
		
		A related concept, {\it timbre}, is introduced by Emilio Pisanty in his Math Stack Exchange question, ``Can you hear the shape of a drum by choosing where to drum it?" \cite{MSE}.  Pisanty defines the {timbre} of a drum at a point $x$ as the sequence of values  $$\sqrt{\sum_{\lambda_j=\lambda} |e_j(x)|^2},$$ where $\lambda$ ranges over distinct eigenvalues of $\sqrt{-\Delta_g}$. 
		
		These notions of homophonicity and timbre can certainly be extended to general Riemannian manifolds, with or without boundaries.
		
		\begin{definition}\label{homophonic} Two pointed (with basepoints in the interior) compact connected Riemannian manifolds  $(M_1,g_1,x_1)$ and $(M_2,g_2,x_2)$ are deemed homophonic if 
			$N_{x_1}(\lambda)=N_{x_2}(\lambda)$ for all $\lambda\ge0$. 
		\end{definition}

		We remark that our definition is slightly different from that in \cite{Conway} because we consider the values of $N_x$ instead of individual eigenfunctions. Nevertheless, as we will demonstrate later, two pointed manifolds that are homophonic in the sense of Definition \ref{homophonic} imply that they, when considered as drums, produce identical sounds in a very strong sense when struck at basepoints $x_1$ and $x_2$. In other words, they share the same timbre as defined in \cite{MSE}.
		
		A natural question to pose is: given {\it one} Riemannian manifold $(M,g)$, if $(M,g,x_1)$ and $(M,g,x_2)$ are homophonic, must $x_1$ be essentially the same as $x_2$?  Specifically, must there exist an isometry $i:M\to M$ such that $i(x_1)=x_2$? This question is mathematically equivalent to our Question \ref{echolocation question}.

		From a physical standpoint, this gives us another equivalent formulation of our echolocation problem: Given a drum of known shape, if it is struck at a point $x$, can you determine $x$ up to isometry by listening to the sounds the drum produces?
		
		For completeness, we will now argue that $N_x$ determines the sound of a drum when it is struck at point $x$ in a very strong sense. Suppose a drum is modeled by a Riemannian manifold $(M,g)$, and we hit this drum with a drumstick at a point $x\in M$, the drum will then vibrate. The vibration of the drum can be described mathematically by the solution of the following wave equation:\footnote{Perhaps it would be more physically realistic to take initial conditions $u(0) = 0$ and $\dot u(0) = \delta_x$, but this does not change the argument.}
		\[
		(\partial_t^2 - \Delta_g)u = 0 \qquad \text{ with } \qquad \begin{cases}
			u(0) = \delta_x \\
			\dot u(0) = 0.
		\end{cases}
		\]
		
		Using the wave kernel on $M$ we can explicitly write down $u$ as
		\begin{align*}
			u(t,y)&=\int_M\cos(t\sqrt{-\Delta_g})(y,z)\delta_x(z)dz\\
			&=\cos(t\sqrt{-\Delta_g})(y,x)=\sum_{\lambda_j}\cos(t\lambda_j)\overline {e_j(x)} e_j(y).
		\end{align*}
		Thus we will be hearing sound composed of various frequencies $\lambda_j$ each with a possibly different volume. For a given frequency $\lambda$ the sound we hear is
		\[
		\sum_{\lambda_j=\lambda}\cos(t\lambda_j)\overline {e_j(x)} e_j(y).
		\]
		The volume of this sound is proportional to the maximal $L^2_y(M)$ norm of the function above, which is the square root of
		\[
		\sum_{\lambda_j=\lambda} |e_j(x)|^2=N_x(\lambda)-\lim_{\epsilon\to 0}N_x(\lambda-\epsilon).
		\]
		To sum up, by striking the drum $M$ at the point $x$, we will be able to hear the above quantities for all values of $\lambda$, which is equivalent to hearing the function $N_x(\lambda)$.
		
		Lastly, we remark that the example of Buser, Conway, Doyle, and Semmler also serves as a counterexample to the echolocation problem if we omit the connectedness assumption.
		
		\subsection{Can you locate yourself using  Brownian motion?}
		In this section, we present a physical interpretation of the echolocation problem (Question \ref{echolocation question}), using the concept of a heat kernel.
		
		Brownian motion on a compact Riemannian manifold $(M,g)$ (with Neumann boundary condition if $\partial M\neq \emptyset$), can be defined as a Markov process with its transition density function being $p(t, x, y)$. This function $p(t,x,y)$ represents the heat kernel associated with the one half the Laplace--Beltrami operator. For an in-depth discussion on this, refer to the lecture notes \cite{hsu}.
		
		For each $t > 0$, the pointwise Weyl counting function $N_x$ can be linked to the heat kernel as follows:
		\[
		\int_{-\infty}^\infty e^{-t\lambda^2/2} dN_x(\lambda) = \sum_j e^{-t\lambda_j^2/2} e_j(x) \overline{e_j(x)} = e^{\frac 12t\Delta_g}(x,x)=p(t,x,y)|_{y=x}.
		\]
		This means that a complete understanding of $N_x$ as a distribution in $\lambda$ leads to a complete understanding of $p(t,x,x)$. According to Lerch's theorem, the reverse is also true.
		
		From a physical perspective, consider a situation where a particle is released at a specific point $x\in M$ at time zero and is permitted to engage in Brownian motion within $M$. The term $p(t,x,x)$ here represents the probability density function of the particle returning to the initial point $x$ at a specific time $t$. Hence, the heat kernel interpretation of Question \ref{echolocation question} is as follows: If you find yourself in a familiar room where movement is restricted, could you determine your precise location by releasing particles at time 0 and observing their probability of returning to your position at any given time $t$ during a Brownian motion experiment?

		{
			We remark also that
			\[
			e^{\frac 1 2 t \Delta_g}(x,x) = u(t,x)
			\]
			where $u$ solves the initial value problem
			\[
			\partial_t u - \frac 1 2 \Delta_g u = 0 \qquad u(0) = \delta_x.
			\]
			The audible quantity $e^{\frac 1 2 t \Delta_g}(x,x)$ can then be viewed as follows: Take a unit of thermal energy and concentrate it at $x$, but then let it diffuse over the manifold. $e^{\frac 1 2 t \Delta_g}(x,x)$ is then the temperature of $M$ at $x$ after time $t$. If we know this temperature for all $t$, can we deduce $x$ up to symmetry?}
		
		\subsection{Schr\"odinger interpretations}\label{Schrodinger}
		
		{Just as there are two interpretations associated with the wave equation, there are two distinct physical interpretations for Question \ref{echolocation question} when considering the Schr\"odinger equation. This stems from the fact that both the wave and Schr\"odinger equations conserve some kind of $L^2$ energy.}
		
		\subsubsection{Quantum superposition}
		
		One can model the evolution of a strictly confined particle at a point $x$ at time $0$, by solving the following Schr\"odinger equation:
		\[
		(i\partial_t - \Delta_g)u = 0 \qquad \text{ with } \qquad {u(0) = \delta_x.}
		\]
		By the Schr\"odinger kernel on $M$, we can explicitly express $u$ as
		\[
		u(t,y)=\int_Me^{it\Delta_g}(y,z)\delta_x(z)dz=e^{-it\Delta_g}(y,x)=\sum_{\lambda_j}e^{it\lambda_j^2}\overline {e_j(x)} e_j(y).
		\]
		Therefore, the resulting solution is the weighted superposition of eigenstates $e_j$. However, the total $L^2$ norm of $u$ is infinite, prohibiting us from interpreting everything probabilistically. We overcome this limitation by choosing a natural approximation of $\delta_x$. Indeed, for each $\Lambda\in\mathbb N,$ we define
		\[\delta_{x,\Lambda}(y):=\sum_{\lambda_j\le \Lambda}e_j(x)\overline{e_j(y)}.\]
		Then it is clear that $\delta_{x,\Lambda}\to\delta_x$ as $\Lambda\to\infty$ in the sense of distribution. We interpret $\delta_{x,\Lambda}$ as an approximation to a particle strictly confined at $x$. If we solve
		\[
		(i\partial_t - \Delta_g)u_\Lambda = 0 \qquad \text{ with } \qquad u_\Lambda(0) = \delta_{x,\Lambda}.
		\]
		we get
		\[
		u_\Lambda(t,y)=\int_Me^{it\Delta_g}(y,z)\delta_{x,\Lambda}(z)dz=\sum_{\lambda_j\le \Lambda}e^{-it\lambda_j^2}\overline {e_j(x)} e_j(y).
		\]
		Therefore, the probability of finding the resulting particle with energy less or equal to $\lambda^2$ is given by
		\[P_\Lambda(\lambda):=\dfrac{N_x(\lambda)}{N_x(\Lambda)}.
		\]
		Note that we have the pointwise Weyl law (See e.g. \cite{SFIO})
		\[
		N_x(\Lambda)=\frac{\omega_d}{(2\pi)^{d}}\Lambda^d+O(\Lambda^{d-1}),
		\]
		where $\omega_d=\vol B^d$ denotes the volume of the  unit ball in $\mathbb R^d$.
		
		Therefore, we can recover $N_x(\lambda)$ from the probability $P_\Lambda(\lambda)$ by the formula
		\[
		N_x(\lambda)=P_\Lambda(\lambda)N_x(\Lambda)
		=\frac{\omega_d}{(2\pi)^{d}}P_\Lambda(\lambda)\Lambda^d+O(\lambda^d\Lambda^{-d}\Lambda^{d-1})=\frac{\omega_d}{(2\pi)^{d}}P_\Lambda(\lambda)\Lambda^d+o(1).
		\]
		Here we have chosen $\Lambda$ large enough so that $\lambda\ll \Lambda^\frac1d$.
		
		In physical terms, if we release a strictly confined quantum particle at the point $x$, then it immediately becomes the superposition of various quantum states. If we repeat this experiment, we can observe all the conditional probabilities of finding it in an eigenstate with energy $\lambda^2$ less or equal to a sufficiently large energy cap $\Lambda^2$. Can we then infer the location of $x$?
		
		\subsubsection{Probability density associated with an energy level}
		
		There is another quantum mechanical interpretation of this problem. Given a compact manifold $(M,g)$, we may ask where a quantum particle constrained to a single energy level $E$ is likely to be found. Assuming the potential function is zero, such a particle is represented by a function $\psi$ solving
		\[
		\Delta_g \psi + E \psi = 0, \qquad \| \psi \|_{L^2(M)} = 1.
		\]
		Given our usual orthonormal basis $e_1, e_2, \ldots$ of eigenfunctions, we may characterize $\psi$ as a linear combination
		\[
		\sum_{\lambda_j^2 = E} z_j e_j \qquad \text{ where } \qquad \sum_{\lambda_j^2 = E} |z_j|^2 = 1.
		\]
		The probability density function for the particle is given by $|\psi|^2$. If the multiplicity $m_E$ of the eigenspace corresponding to energy level $E$ is greater than $1$, we have some freedom in our choice of coefficients $z_j$. Keeping in the spirit of the problem, we take the coefficients $(z_j)_{\lambda_j^2 = E}$ to be drawn randomly and uniformly from the unit sphere in $\C^{m_E}$. The probability density function of our particle over $M$ is then
		\begin{align*}
			\E |\psi|^2 = \E \left| \sum_{\lambda_j^2 = E} z_j e_j \right|^2 = \E \sum_{\lambda_j^2 = \lambda_k^2 = E} z_j \overline{z_k} e_j \overline{e_k} = \frac{1}{m_E} \sum_{\lambda_j^2 = E} |e_j|^2,
		\end{align*}
		where in the last step we have used
		\[
		\E[z_j \overline{z_k}] = \begin{cases}
			\frac{1}{m_E} & \text{ if } j = k \\
			0 & \text{ if } j \neq k.
		\end{cases}
		\]
		Since the manifold $(M,g)$ is familiar to us, we also know the multiplicities of the Laplace eigenvalues. Hence, the probability density function above evaluated at $x$ may be written as 
		\[
		\dfrac{\displaystyle N_x(\lambda) - \lim_{\epsilon \to 0}N_x(\lambda - \epsilon)}{\#\{j : \lambda_j = \lambda\}}
		\]
		where $\lambda^2 = E$.
		
		Now we are ready to give the final physical interpretation of Question \ref{echolocation question}. Given the probability (density) of finding a typical quantum particle with energy $E$ at $x$, for all possible values of $E$, can we determine the location of $x$?} 
	
	\section{Examples and audible quantities} \label{sec: examples}
	
	{ 
		For the sake of precision, we define an \emph{audible quantity} to be any function $f$ on $M$ satisfying $f(x) = f(y)$ whenever $N_x = N_y$ identically. An audible quantity is then completely determined by the pointwise counting functions. Note, $f$ can take any kind of object as a value---including sets and functions---as opposed to only numbers.
	}
	{ Here and throughout, we will say that {\it echolocation holds} on a manifold if Question \ref{echolocation question} is answered in the affirmative for that manifold.}
	
	\begin{example}[A string with fixed ends] Our first example should be a simple string $M = [0,a]$ equipped with Dirichlet boundary conditions. Here, the first eigenfunction reads 
		\[
		e_1(x) = \sqrt{\frac{2}{a}} \sin(\pi x/a),
		\]
		and we note that the audible quantity
		\[
		N_x(\pi/a) = |e_1(x)|^2 = \frac{2}{a} \sin^2(\pi x /a)
		\]
		alone is enough to deduce the location of $x$ up to a reflection about the midpoint of the string.
	\end{example}
	
	{\begin{example}[A rectangular plate]
			Consider a rectangular plate $M = [0,a] \times [0,b]$ with Dirichlet boundary conditions. By rescaling, we may assume that $a=1$ and $0<b\le 1$. A complete orthonormal basis of eigenfunctions can be written down as
			\[
			e_{n,m}(x,y) = \frac{2}{\sqrt{b}} \sin(\pi n x) \sin(\pi m y/b), \qquad n,m = 1,2,\ldots
			\]
			with respective eigenvalues
			\[
			\lambda_{n,m} = \pi \sqrt{n^2 + \frac{m^2}{b^2}}.
			\]
			We first handle the case when $0<b<1.$ In this case, the first two eigenvalues $\lambda_{1,1}$ and $\lambda_{2,1}$ are both simple. We consider the audible quantity
			\begin{equation}
				\label{eq: audible} 
				\frac{N_{(x,y)}(\lambda_{2,1}) - N_{(x,y)}(\lambda_{1,1})}{N_{(x,y)}(\lambda_{1,1})} = \frac{\sin^2(2\pi x) \sin^2(\pi y/b)}{\sin^2(\pi x) \sin^2(\pi y / b)} = 2 \cos^2(\pi x),
			\end{equation}
			which is  enough to identify the $x$-coordinate up to symmetry. We can then determine the $y$-coordinate up to symmetry by
			\[
			\sin^2(\pi y/b)= \frac{N_{(x,y)}(\lambda_{1,1})}{1-\cos^2(\pi x)}.
			\]
			Now if $b=a=1,$ our domain becomes a square, which has an additional mirror symmetry along the diagonal. As a result, $\lambda_{2,1}=\lambda_{1,2}$ has multiplicity 2.
			The same audible quantity \eqref{eq: audible} now reads
			\begin{equation*}
				\label{eq: audible} 
				\frac{N_{(x,y)}(\lambda_{2,1}) - N_{(x,y)}(\lambda_{1,1})}{N_{(x,y)}(\lambda_{1,1})} =  2 (\cos^2(\pi x)+\cos^2(\pi y))=  4-2(\sin^2(\pi x)+\sin^2(\pi y)).
			\end{equation*}
			Noting $N_{(x,y)}(\lambda_{1,1})=\sin^2(\pi x)\cdot\sin^2(\pi y)$, we see that both the product and sum of $\sin^2(\pi x)$ and $\sin^2(\pi y)$ are  audible. Solving the associated quadratic equation concludes that echolocation holds on this square plate.
			
			{We see an immediate jump in complexity---much of it number-theoretic---when we consider hyperrectangles. By arguments similar to the $0<b<1$ case above, one can see that echolocation holds for hyperrectangle with distinct side lengths. However, more general cases seem much more difficult in higher dimensions.}
	\end{example}}

	We now try to extract what we can from the heat kernel, though we do not venture past the first two terms. We have for each $t > 0$
	\[
	\int_{-\infty}^\infty e^{-t\lambda^2} dN_x(\lambda) = \sum_j e^{-t\lambda_j^2} e_j(x) \overline{e_j(x)} = e^{t\Delta_g}(x,x).
	\]
	It is well-known (e.g., \cite{heatsecondterm}) that for manifolds without boundary, the heat kernel has a local asymptotic expansion
	\[
	(4\pi t)^{n/2} e^{t\Delta_g}(x,x) = 1 + \frac{t}{3} K(x) + O(t^2), \qquad 0 < t \ll 1,
	\]
	where $K(x)$ denotes the scalar curvature of $M$ at $x$. We conclude:
	
	\begin{proposition}\label{curvature} The scalar curvature at $x$ is an audible quantity.
	\end{proposition}
	
	This audible quantity is enough, for example, to show that echolocation is possible on an ellipsoid of revolution.
	
	\begin{example}[An ellipsoid of revolution.]Let $M$ be a non-spherical ellipsoid of revolution, e.g., the surface
		\[
		x^2 + y^2 + \frac{z^2}{a^2} = 1 
		\]
		in $\R^3$ with $0 < a \neq 1$. Recall that the Gaussian curvature of a surface embedded in $\R^3$ coincides with the sectional and also the scalar curvature. It is a routine yet somewhat tedious calculation to verify the Gaussian curvature of this ellipsoid is given by
		\[
		K(x,y,z) = \frac{a^2}{(z^2(1 - a^{-2}) + a^2)^2}.
		\]
		This audible quantity is enough to determine the $z$-coordinate of a point up to sign, which in turn determines a point up to symmetry. {We remark that similar arguments work on a number of surfaces of revolution. For example, one can easily show that echolocation holds on a torus generated by revolving a circle in $\mathbb R^3$.}
	\end{example}
	
	Next, we extract what we can from the wave equation. We present here a brief wavefront set calculation that arises in some form or another in most results about pointwise asymptotics. For examples of this kind of calculation, see \cite{DG, SFIO, Hang, SZDuke}. For background on the calculus of wavefront sets and microlocal analysis, we refer the reader to \cite{DuistermaatFIOs, HormanderPaper, SFIO}.
	
	Recall the identity
	\[
	\int_{-\infty}^\infty e^{it\lambda} \, dN_x(\lambda) = e^{it\sqrt{-\Delta_g}}(x,x).
	\]
	We interpret the right side as the distribution in $t$ given by the composition $U \circ (\delta_x \otimes \delta_x)$ where $U$ is an operator with distribution kernel $U(t,x,y) = e^{it\sqrt{-\Delta_g}}(x,y)$. By the results in \cite[Chapter 29]{HormanderIV}, the wavefront relation of $U$ is given by
	\[
	\WF'(U) = \{(t,p(x,\xi) ; x, -\xi, G^t(x,\xi)) : t \in \R, \ (x,\xi) \in \dot T^*M \}
	\]
	where
	\[
	p(x,\xi) = |\xi|_{g(x)} = \left( \sum_{i,j = 1}^n g^{ij}(x) \xi_i \xi_j \right)^{\frac12}
	\]
	is the principal symbol of the Laplace--Beltrami operator and $G^t$ is the time-$t$ homogeneous geodesic flow on $\dot T^*M$. For an introduction to wavefront sets and their calculus, see \cite{DuistermaatFIOs} and \cite{HormanderPaper}. The wavefront set of $\delta_x \otimes \delta_x$ is, conveniently, the twofold product of the wavefront set of $\delta_x$, i.e.
	\[
	\WF(\delta_x\otimes \delta_x) = \{(x,\xi, x, \eta) : \xi,\eta \in \dot T_x^* M \}.
	\]
	The composition then satisfies
	\begin{align*}
		\WF(U \circ (\delta_x \otimes \delta_x)) &\subset \WF'(U) \circ \WF(\delta_x \otimes \delta_x) \\
		&= \{ (t, \tau) : \tau > 0 , \ \exists \xi,\eta \in \dot T_x^*M \text{ with } G^t(x,\xi) = (x,\eta) \}.
	\end{align*}
	
	In other words, the $t$-singular support of $e^{it\sqrt{-\Delta_g}}(x,x)$ can only occur at times equal to the length of a \emph{looping geodesic}, a unit-speed geodesic on $M$ which departs $x$ and arrives back at $x$ after time $t$. We will call such times the \emph{looping times at $x$}.

	\begin{proposition}\label{looping times} The $t$-singular support of $e^{it\sqrt{-\Delta_g}}(x,x)$ is audible and contained in the set of looping times at $x$.
	\end{proposition}
	
	We remark that in most examples, the looping times set will be completely determined by the $t$-singular support of $e^{it\sqrt{-\Delta_g}}(x,x)$ and hence audible. Only when specific well-arranged destructive interference occurs will some looping times disappear from the $t$-singular support of $e^{it\sqrt{-\Delta_g}}(x,x)$.
	
	If $M$ has a smooth boundary and the eigenfunctions satisfy Dirichlet or Von-Neumann boundary conditions, then Proposition \ref{looping times} still holds, except here we call the trajectories billiard trajectories, and they reflect off of the boundary in the expected way. 
	Moreover, in certain cases, even partial knowledge of the looping times at $x$ is enough to determine $x$ up to symmetry. For instance, as we have shown in Remark \ref{disk example}, the shortest looping time at a point $x$ in the circular disk $D$ equals twice the distance the point is from the boundary, which is enough to determine $x$ up to rotation. In this case, echolocation can also be achieved by examining the first eigenfunction directly.
	
	\section{Generic injections} \label{sec: generic injections}
	
	Let $\mathcal G$ be a second countable Baire topological space and $X$ a finite-dimensional $C^1$ manifold (always assumed to be Hausdorff and second countable). Given a closed subset $\mathcal H \subset \mathcal G \times X$, we seek to find sufficient conditions under which the image of $\mathcal H$ through the projection $\mathcal G \times X \to \mathcal G$ is meager.
	
	To see how this can help us show a class of maps is injective, let $M$ be a smooth $n$-dimensional compact manifold, let $\Delta$ denote the diagonal in $M \times M$, and consider the case where $\mathcal G = C^\infty(M, \R^{2n+1})$, where $X = (M \times M) \setminus \Delta$, and where $\mathcal H = \{(f,x,y) : f(x) = f(y)\}$. The projection of $\mathcal H$ onto $\mathcal G$ is then
	\[
	\{f \in C^\infty(M, \R^{2n+1}) : f(x) = f(y) \text{ for some } x \neq y\}.
	\]
	If we can show this set is meager, then its complement, the set of smooth injective maps $M \to \R^{2n+1}$, is residual in $C^\infty(M, \R^{2n+1})$. Once we have established our tools, we will prove generic versions of the Whitney immersion theorem and weak Whitney embedding theorem for compact manifolds as illustrative examples.
	
	\subsection{The main tool}

	\begin{definition}\label{def: slice}
		Take $\mathcal G$, $X$, and $\mathcal H$ as above.
		Fix $(g_0, x_0) \in \mathcal H$ and let $n$ be a positive integer. Let $U$ be some neighborhood of the origin in $\R^n$ and let $V$ be a neighborhood of $x_0$ in $X$. Suppose:
		\begin{enumerate}
			\item $\rho : U \to \mathcal G$ is a continuous map with $\rho(0) = g_0$.
			\item $\pi : \rho(U) \times V \to \R^n$ is a continuous map with $\pi(\mathcal H) = \{0\}$.
			\item The map $\Phi : U \times V \to \R^n$ given by
			\[
			\Phi(s,x) = \pi(\rho(s), x)
			\]
			is $C^1$ on $U \times V$.
			\item We have
			\[
			\det d_s \Phi(0,x_0) \neq 0.
			\]
		\end{enumerate}
		Then, we say the map $\Phi$ is an $n$-dimensional \emph{slice} across $\mathcal H$ at $(g_0,x_0)$.
	\end{definition}

	\begin{theorem}\label{thm: meager}
		Let $\mathcal G$ be a second countable Baire topological space, let $X$ be a finite-dimensional $C^1$ manifold, and let $\mathcal H \subset \mathcal G \times X$ be a closed subset. 
		If there exists an $n$-dimensional slice across $\mathcal H$ at each $(g_0,x_0) \in \mathcal H$ with $n > \dim X$, then
		\[
		\{g \in \mathcal G : (g,x) \in \mathcal H \text{ for some } x \in X \}
		\]
		is meager in $\mathcal G$. If $X$ is compact, then this set is closed and nowhere dense.
	\end{theorem}
	
	\begin{proof}
		We claim we can cover $\mathcal H$ by open neighborhoods $W \times V$ such that the projection of $\mathcal H \cap (W \times V)$ onto $\mathcal G$ has nowhere dense closure. Since both $\mathcal G$ and $X$ are second countable, we may select a countable subcover. Taking the countable union of these meager sets in $\mathcal G$ yields another meager set, and the first part of the theorem will have been proved. If $X$ is compact, then the projection of $\mathcal H$ onto $\mathcal G$ is closed, from which the second part of the theorem follows.
		
		To prove the claim, fix $(g_0, x_0) \in \mathcal H$ and fix an open neighborhood $W$ about $g_0$. By assumption, there exists an $n$-dimensional slice $\Phi : U \times V \to \R^n$ across $\mathcal H$ at $(g_0,x_0)$. After perhaps shrinking $U$, we ensure the image of $\rho : U \to \mathcal G$ as in Definition \ref{def: slice} is contained in $W$.
		
		Using parts (2) and (3) of Definition \ref{def: slice}, and after perhaps shrinking both $U$ and $V$ further, the implicit function theorem allows us to write
		\[
		\Phi^{-1}(0) = \{(s(x), x) : x \in V\}
		\]
		where $s(x)$ is a $C^1$ function of $x$. By Sard's theorem, the image of $s : V \to U$ is measure zero in $U$. If we allow ourselves to shrink $V$ even more, we can take the closure of the image of $s : V \to U$ to be compact. Pushing this compact nowhere dense set through $\rho : U \to W$ yields a compact nowhere dense subset of $\mathcal G$. Furthermore, this set contains the projection of $\mathcal H \cap (W \times V)$ onto $\mathcal G$ by construction. This concludes the proof of the claim and of the theorem.
	\end{proof}
	
	\subsection{Generic immersions and embeddings} We now use the tools above to prove a generic variant on Whitney's weak embedding theorem for compact manifolds. The purpose is illustrative. These quick arguments model how we will use these tools in our main result.
	
	We start by proving a generic version of Whitney's immersion theorem.
	
	\begin{proposition}[Generic Whitney immersion for compact manifolds] Let $M$ be a smooth, compact, finite-dimensional manifold. Then, the immersions in $C^\infty(M,\R^n)$ with $n > 2\dim M - 1$ form an open dense set.
	\end{proposition}
	
	\begin{proof} Fix any Riemannian metric on $M$ and let $SM$ denote the unit sphere bundle. Let 
		\[
		\mathcal H = \{(f,x,v) \in C^\infty(M,\R^n) \times SM : df(x)v = 0\}.
		\]
		We claim we can produce an $n$-dimensional slice across $\mathcal H$ at any point, after which we are done by Theorem \ref{thm: meager}.
		
		Fix $(f_0,x_0,v_0)$. Take $\pi : C^\infty(M,\R^n) \times SM \to \R^n$ to be given by $\pi(f,x,v) = df(x)v$. Let $\beta \in C^\infty(M, \R)$ be a smooth function with $\beta(x_0) = 0$ and $d\beta(x_0)v_0 = 1$. Then, take $\rho : \R^n \to \R^n$ with
		\[
		\rho(s) = f_0 + \sum_{i = 1}^n s_i \beta e_i
		\]
		and note
		\[
		\Phi(s,x,v) = \pi(\rho(s), x, v) = df_0(x)v + \sum_{i = 1}^n s_i (d\beta(x)v) e_i
		\]
		and hence
		\[
		\Phi(s,x_0,v_0) = df_0(x_0)v_0 + \sum_{i = 1}^n s_i e_i \qquad \text{ and } \qquad d_s \Phi(0, x_0, v_0) = I.
		\]
		We are done after we restrict to suitable neighborhoods $U$ and $V$ in $\R^n$ and $SM$, respectively.
	\end{proof}
	
	\begin{proposition}[Generic weak Whitney embedding for compact manifolds] Let $M$ be a smooth, compact, finite-dimensional manifold. Then, the embeddings in $C^\infty(M,\R^n)$ for $n > 2\dim M$ form an open, dense set.
	\end{proposition}
	
	\begin{proof}
		We will produce suitable slices across the set
		\[
		\mathcal H = \{(f,x,y) \in C^\infty(M,\R^n) \times (M \times M) \setminus \Delta : f(x) = f(y) \}
		\]
		and show by Theorem \ref{thm: meager} that the set of injections in $C^\infty(M,\R^n)$ is residual. We first establish $\pi(f,x,y) = f(x) - f(y)$. Then, we fix $(f_0,x_0,y_0) \in \mathcal H$ and let $\beta$ be a bump function on $M$ for which $\beta(x_0) = 1$ and $\beta(y_0) = 0$. Take $\rho : \R^n \to C^\infty(M,\R^n)$ with
		\[
		\rho(s) = f_0 + \sum_{i = 1}^n s_i \beta e_i.
		\]
		Then,
		\[
		\Phi(s,x,y) = \pi(\rho(s),x,y) = f_0(x) - f_0(y) + \sum_{i = 1}^n s_i (\beta(x) - \beta(y)) e_i
		\]
		and
		\[
		\Phi(s,x_0,y_0) = s \qquad \text{ and } \qquad d_s\Phi(s,x_0,y_0) = I.
		\]
		We are done after selecting appropriate neighborhoods $U$ and $V$ in $\R^n$ and $M \times M \setminus \Delta$, respectively.
		
		We have just shown that the set of smooth injections in $C^\infty(M,\R^n)$ is residual. Hence, the set of injective immersions is residual in $C^\infty(M,\R^n)$ by the previous proposition. This is precisely the set of embeddings since $M$ is compact. One quickly verifies the set of embeddings in $C^\infty(M,\R^n)$ is open, and the proposition follows.
	\end{proof}

	\section{Proof of Theorem \ref{thm: generic echolocation}} \label{sec: proof of main}
	
	\subsection{The audible objects} Before proceeding, we will take a moment to establish a very clear connection between the solution operator $\cos(t\sqrt{-\Delta_g})$ for the wave equation on $(M,g)$ and the ``audible'' object
	\[
	\int_{-\infty}^\infty \cos(t\lambda) \, dN_x(\lambda) = \sum_j \cos(t\lambda_j) |e_j(x)|^2
	\]
	as a distribution in $t$. The difficulty lies in the distinction between smooth functions and smooth densities on $M$. We must be careful in this regard since we will be varying the metric and hence the natural volume density on $M$.
	
	We observe that, given smooth initial data $f$ on $M$, the function
	\[
	u(t,x) = \sum_j \cos(t\lambda_j) \left(\int_M f(y) \overline{e_j(y)} \, dV_g(y)\right) e_j(x)
	\]
	solves the initial value problem
	\[
	\partial_t^2 u - \Delta_g u = 0 \quad \text{ with } \quad 
	\begin{cases}
		u(0,x) = f(x) \\
		\partial_t u(0,x) = 0
	\end{cases},
	\]
	and hence we write
	\[
	u(t,x) = \cos(t\sqrt{-\Delta_g})f(x).
	\]
	The kernel of the solution operator can then be written in local coordinates as
	\[
	\cos(t\sqrt{-\Delta_g})(x,y) = \sum_j \cos(t\lambda_j) e_j(x) \overline{e_j(y)} |g(y)|^{1/2}.
	\]
	Restricting to the diagonal and interpreting the result as a distribution in $t$, we find that
	\begin{equation}\label{eq: renormalized wave}
		\varphi(g,t,x) := \sum_j \cos(t\lambda_j) |e_j(x)|^2 = |g(x)|^{-1/2} \cos(t\sqrt{-\Delta_g})(x,x)
	\end{equation}
	is the relevant audible quantity. Note the renormalization by the volume element. We observe that, since $N_x$ is identically $0$ on the negative real line and real-valued otherwise, we may recover $N_x$ from its cosine transform, namely from $\varphi(g,t,x)$. We conclude:
	
	\begin{proposition} \label{prop: cosine transform}
		For fixed metric $g$ and for $x,y \in M$, we have $N_x \equiv N_y$ if and only if
		\[
		\varphi(g,t,x) = \varphi(g,t,y) \qquad \text{ for all } t \in \R,
		\]
		where equality is understood in the sense of distributions in $t$.
	\end{proposition}
	
	Our main tool for establishing genericity of metrics for which Theorem \ref{thm: generic echolocation} holds is Theorem \ref{thm: meager}. To this end, we need to understand the partial derivative of $\varphi(g,t,x)$ with respect to some variation of $g$.
	
	\subsection{A directional derivative formula}
	
	Let $\mathcal G$ denote the set of Riemannian metrics on a compact manifold $M$. It is well-known that $\mathcal G$ is a manifold modeled on a Fr\'echet space (see, e.g., \cite[Section 1.1]{blair2000spaces}).
	In fact since $M$ is compact, $\mathcal G$ is separable and hence second countable. Fix a metric $g \in \mathcal G$ and a smooth, real-valued function $h$ on $M$. We identify $h$ with a vector in $T_g \mathcal G$ to act on $C^1$ functions by
	\[
	D_h f(g,x) = \left.\frac{d}{ds}\right|_{s = 0} f(e^{sh}g,x).
	\]
	Our objective now is to find a workable formula for $D_h \varphi$ with $\varphi$ as in \eqref{eq: renormalized wave}. This requires that we study a slightly different object. Fix a smooth function $f$ on $M$, which will be specified later, and take
	\[
	u(g,t,x) = \cos(t\sqrt{-\Delta_g})f(x) = \int_M \cos(t\sqrt{-\Delta_g})(x,y) f(y) \, dV_g(y),
	\]
	the solution operator of the wave equation
	\begin{equation}\label{wave equation for u}
		\partial_t^2 u - \Delta_g u = 0 \quad \text{ with } \quad \begin{cases}
			u(g,0,x) = f(x) \\
			\partial_t u(g,0,x) = 0
		\end{cases}.
	\end{equation}
	
	We suppose for a moment that $u(g,t,x)$ is sufficiently differentiable in all variables. (We will address this assumption in Lemma \ref{lem: C^1} below.) Applying the directional derivative to the homogeneous wave equation above, we see $D_h u$ solves the nonhomogeneous wave equation
	\begin{equation}\label{eq: nonhomogeneous wave}
		\partial_t^2 {v} - \Delta_g {v} = [D_h, \Delta_g]u \quad \text{ with } \quad \begin{cases}
			{v} = 0 \\
			\partial_t {v} = 0
		\end{cases}.
	\end{equation}
	In order to extract useful information about $D_h u$ from this equation, we require the following identity.
	
	\begin{lemma}\label{lem: commutator}
		\[
		[D_h, \Delta_g] = -h \Delta_g + \left(\frac n 2 - 1 \right) \nabla_g h.
		\]
	\end{lemma}
	
	\begin{proof}
		First, we write
		\begin{align*}
			\Delta_{e^{sh} g} &= |e^{sh} g|^{-1/2} \sum_{i,j} \partial_i(|e^{sh}g|^{1/2} e^{-sh} g^{ij} \partial_j) \\
			&= e^{-\frac n 2 sh} |g|^{-1/2} \sum_{i,j} \partial_i(e^{(\frac n 2 - 1) sh} |g|^{1/2} g^{ij} \partial_j) \\
			&= e^{-sh} |g|^{-1/2} \sum_{i,j} \partial_i(|g|^{1/2} g^{ij} \partial_j) + \left( \frac n 2 - 1 \right) s e^{-sh} \sum_{i,j} g^{ij} \partial_i h \partial_j \\
			&= e^{-sh} \Delta_g + \left( \frac n 2 - 1 \right) s e^{-sh} \nabla_g h .
		\end{align*}
		It follows that, if $f$ is a smooth function on $\mathcal G \times M$,
		\begin{align*}
			D_h \Delta_g f(g,x) &= \left.\frac{d}{ds}\right|_{s = 0} \Delta_{e^{sh}g} f(e^{sh}g, x) \\
			&= \left.\frac{d}{ds}\right|_{s = 0} \left( e^{-sh} \Delta_g f(e^{sh}g,x) + \left( \frac n 2 - 1 \right) s e^{-sh} (\nabla_g h) f(e^{sh}g,x) \right) \\
			&= \Delta_g D_h f(g,x) - h \Delta_g f(g,x) + \left(\frac n 2 - 1 \right) (\nabla_g h) f(g,x).
		\end{align*}
		The lemma follows.
	\end{proof}
	
	Lemma \ref{lem: commutator} will be used in many ways. However, its present use comes from the miraculous fact that $[D_h, \Delta_g]$ is a second-order differential operator acting in the spacial variable only. We use this to exploit the fact that $u(g,t,x)$ is smooth in the $x$ variable. In particular, the forcing term $[D_h, \Delta_g]u$ in \eqref{eq: nonhomogeneous wave} is, for fixed $g$, a smooth function of $t$ and $x$. Hence any solution to \eqref{eq: nonhomogeneous wave} is also smooth in $t$ and $x$.
	
	Let us rephrase what we have found so far. We have just shown that $u(e^{sh}g, t, x)$, as a function in $(s,t,x)$, has a distributional derivative $\partial_s u$ which satisfies \eqref{eq: nonhomogeneous wave}, and hence is also a function which is smooth in $t$ and $x$. We will show that both $u$ and its distributional derivative $\partial_s u$ are continuous functions in $s$. This will force $u$ to be a $C^1$ function in $s$.
	
	\begin{lemma}\label{lem: C^1}
		$u(e^{sh}g,t,x)$ from \eqref{wave equation for u} is a $C^1$ function on $(s,t,x) \in \R \times \R \times M$, and
		\[
		\partial_s u(e^{sh}g, t, x) = D_h u(g,t,x) \qquad \text{ at } s = 0
		\]
		and solves \eqref{eq: nonhomogeneous wave}.
	\end{lemma}
	
	\begin{proof}First, we show that for fixed $(t,x)$, $u(e^{sh} g, t, x)$ is continuous in $s$. It suffices to show, without loss of generality, that $u(e^{sh} g, t, x)$ is continuous at $s = 0$ for each fixed $t$ and $x$. For the sake of clarity, we will set $w(s,t,x) = u(e^{sh}g, t, x)$ and take
		\[
		v(s,t,x) = w(s,t,x) - w(0,t,x).
		\]
		Note that $v$ satisfies the nonhomogeneous wave equation
		\begin{equation*}\label{eq: 2}
			\partial_t^2 v - \Delta_{g} v = F(s,t,x)
			\quad \text{ with } \quad \begin{cases}
				v(s,0,x) = 0 \\
				\partial_t v(s,0,x) = 0.
			\end{cases}
		\end{equation*}
		Here 
		\begin{equation}
			F(s,t,x) := (\Delta_{e^{sh}g}-\Delta_{g})u(s,t,x)
			=((e^{-sh}-1)\Delta_{g}+(\frac n2-1)se^{-sh}\nabla_{g}h)u(s,t,x)
		\end{equation}
		is a smooth function in $(t,x)$. Its $C^\infty$ semi-norms in $(t,x)$ have limit $0$ as $s\to0$. This is due to the fact that $h$ is fixed, and the $(t,x)$-derivatives of $u$ are bounded uniformly in a compact neighborhood of $s=0$. We conclude from the standard regularity estimates (e.g., \cite[\S 7.2, Theorem 6]{EvansPDE}) and Sobolev embedding that $\lim_{s \to 0} v(s,t,x) = 0$, and so $u(e^{sh}g, t, x)$ is a continuous function in $s$.
		
		The only thing left to show is that the distributional derivative $\partial_s u(e^{sh} g, t, x)$ is continuous in the $s$ variable.  
		The argument is similar. Without loss of generality, we show that $\partial_s u(e^{sh} g, t, x)$ is continuous at $s = 0$ for each fixed $t$ and $x$. We take
		\[
		v_1(s,t,x) = \partial_sw(s,t,x) - \partial_sw(0,t,x).
		\]
		Now we must show $\lim_{s \to 0} v_1(s,t,x) = 0$. Note, $v_1$ satisfies the nonhomogeneous wave equation
		\begin{equation*}
			\partial_t^2 v_1 - \Delta_{g} v_1 = F_1(s,t,x)
			\quad \text{ with } \quad \begin{cases}
				v_1(s,0,x) = 0 \\
				\partial_t v_1(s,0,x) = 0.
			\end{cases}
		\end{equation*}
		Here 
		\begin{align*}
			F_1(s,t,x) :&= [D_h, \Delta_{g}]w(s,t,x)-[D_h, \Delta_{e^{sh}g}]w(0,t,x) +(\Delta_{e^{sh}g}-\Delta_{g})(\partial_s w(s,t,x))
			\\&=[D_h, \Delta_{g}]w(s,t,x)-[D_h, \Delta_{e^{sh}g}]w(0,t,x)\\
			&\hspace{8em} + \left((e^{-sh}-1)\Delta_{g}+\left(\frac n2-1 \right)se^{-sh}\nabla_{g}h \right)(\partial_sw(s,t,x))
		\end{align*}
		is again a smooth function in $(t,x)$, and has vanishing $(t,x)$-$C^\infty$ semi-norms as $s\to0$. We conclude that $\lim_{s\to0}v=0$, which indicate that $\partial_s u$ is continuous in $s$.
	\end{proof}
	
	By Duhamel's principle and \eqref{eq: nonhomogeneous wave}, we have
	\begin{equation}\label{eq: partial h u}
		D_h u(g,t,x) = \int_0^t \frac{\sin((t-s) \sqrt{-\Delta_g})}{\sqrt{-\Delta_g}} [D_h, \Delta_g]u(g,s,x) \, ds
	\end{equation}
	for $t > 0$. We consider an off-diagonal version of $\varphi$,
	\begin{equation}\label{eq: psi cosine}
		\psi(g,t,x,y) = \sum_j \cos(t\lambda_j) e_j(x) \overline{e_j(y)} = |g(y)|^{-1/2} \cos(t\sqrt{-\Delta_g})(x,y).
	\end{equation}
	We make a couple of observations:
	\begin{enumerate}
		\item Since we are free to select a real eigenbasis, we must have $\psi(g,t,x,y) = \psi(g,t,y,x)$ for all $x$ and $y$ in $M$.
		\item $\varphi(g,t,x) = \psi(g,t,x,x)$ as distributions in $t$.
	\end{enumerate}
	We recognize
	\[
	D_h(|g(y)|^{1/2} \psi(g,t,x,y))
	\]
	as the distribution kernel of the solution operator for the nonhomogeneous Cauchy problem \eqref{eq: nonhomogeneous wave} above. Again, we must verify that $\psi(e^{sh}g, t, x, y)$ is $C^1$. In what follows, $\inj(M,g)$ denotes the injectivity radius of the manifold $M$ with metric $g$.

	\begin{lemma}\label{lem: C^1 2}
		For a fixed metric $g$, $\psi(e^{sh}g, t, x, y)$ is $C^1$ at each point in the set
		\[
		\{ (s,t,x,y) : 0 < t < \inj(M,e^{sh}g), \ d_{e^{sh}g}(x,y) < t \}.
		\]
		Furthermore for $d_g(x,y) < t < \inj(M,g)$, we have
		\begin{multline*}
			D_h(|g(y)|^{1/2} \psi(g,t,x,y)) \\
			= \int_0^t \frac{\sin((t-s) \sqrt{-\Delta_g})}{\sqrt{-\Delta_g}}  [D_h, \Delta_g] \cos(s\sqrt{-\Delta_g}) (x,y) \, ds.
		\end{multline*}
	\end{lemma}
	
	\begin{proof}

		The lemma follows if we examine the Hadamard's parametrix of the wave equation closely. Indeed, by \cite[(17.4.6)'']{HormanderIV}, the main term of the parametrix satisfies the lemma since its dependence on the metric is explicit. The remainder solves the inhomogeneous wave equation with an increasingly smooth forcing term. The time and spatial derivatives of the forcing term can be seen to be bounded uniformly in $s$. The lemma then follows from Lemma \ref{lem: C^1}.
	\end{proof}

	It follows by Lemma \ref{lem: C^1 2} that
	\begin{multline*}
		D_h \psi(g,t,x,y) = -\frac{n}{2} h(y) \psi(g,t,x,y) \\
		+ |g(y)|^{-1/2} \int_0^t \frac{\sin((t-s) \sqrt{-\Delta_g})}{\sqrt{-\Delta_g}}[D_h, \Delta_g] \cos(s\sqrt{-\Delta_g})(x,y) \, ds.
	\end{multline*}
	Recalling $\varphi(g,t,x) = \psi(g,t,x,x)$, we obtain
	\begin{multline} \label{eq: D_h phi}
		D_h \varphi(g,t,x) = -\frac{n}{2} h(x) \varphi(g,t,x) \\
		+ |g(x)|^{-1/2} \int_0^t \frac{\sin((t-s) \sqrt{-\Delta_g})}{\sqrt{-\Delta_g}}(x,z) [D_h, \Delta_g] \cos(s\sqrt{-\Delta_g})(z,x) \, dz \, ds,
	\end{multline}
	where $[D_h, \Delta_g]$ acts in the first variable.
	
	By Lemma \ref{lem: commutator}, $[D_h, \Delta_g]$ is supported on $\supp h$. The next proposition describes the support of $D_h \varphi$ in $x$ and $t$.
	
	\begin{proposition}\label{prop: support}
		For fixed $g$, then as a distribution in $x$ and $t$,
		\[
		\supp D_h \varphi \subset \{(t,x) : 2 d_g(x, \supp h) \leq |t|\}.
		\]
	\end{proposition}
	
	\begin{proof}
		We assume $t > 0$. The proposition is proved for negative times in a similar way.
		By Huygens' principle and that $\supp\, [D_h, \Delta_g] \subset \supp h$, we have
		\[
		[D_h, \Delta_g] \cos(s\sqrt{-\Delta_g})(z,y)
		\]
		is supported for $d_g(z,y) \leq |s|$ with $z \in \supp h$. Also by Huygens' principle,
		
		\[
		\frac{\sin((t-s) \sqrt{-\Delta_g})}{\sqrt{-\Delta_g}}(x,z)
		\]
		is supported for $d_g(x,z) \leq |t - s|$. Hence,
		\[
		\int_0^t \frac{\sin((t-s) \sqrt{-\Delta_g})}{\sqrt{-\Delta_g}}  [D_h, \Delta_g] \cos(s\sqrt{-\Delta_g})(x,y) \, ds
		\]
		is supported for $(x,y)$ for which there exists $z \in \supp h$ and $s \in [0,t]$ with $d_g(x,z) \leq t - s$ and $d_g(z,y) \leq s$. If $x = y$, this condition reads as
		\[
		d_g(x,z) \leq \min(t-s, s) \qquad \text{ for some $s$ in $[0,t]$}.
		\]
		The proposition follows after maximizing the right side by taking $s = t/2$. 
		
	\end{proof}
	
	The proof of the proposition has a nice physical interpretation. Suppose we make a small conformal perturbation of the metric $g$ along $e^{sh}g$. The quantity $D_h \varphi(g,t,x)$ tells us the degree to which we can ``hear'' the perturbation if we stand at $x$, clap our hands at time $0$, and listen to the reverberation at time $t$. In order to ``hear'' the change in the metric, the sound must travel from $x$ to the region $\supp h$ at which the metric was perturbed, and then it will have to travel back to $x$.
	
	In order to use the tools in Section \ref{sec: generic injections}, we will need to construct nice perturbations $h$ which satisfy some desirable properties. The following lemma does precisely this. Its proof is rather involved and so deferred to Section \ref{sec: proof lemma}.

	\begin{lemma}\label{lem: h construction}
		Suppose $\dim M = n \geq 2$. Fix $x \in M$ and $t > 0$ less than the injectivity radius of $(M,g)$. There exists a smooth function $h$ supported in the annulus
		\[
		\left\{z \in M : \frac t 3 < d_g(x,z) < \frac{2t}{3} \right\}
		\]
		for which
		\[
		D_h \varphi(g,t,x) \neq 0.
		\]
	\end{lemma}
	
	A standard wavefront set calculation tells us the wavefront set of $D_h \varphi(g,t,x)$, as a distribution in $t$, is supported for $|t|$ outside the injectivity radius of $(M,g)$. Hence, $D_h \varphi(g,t,x)$ is smooth for $|t|$ less than the injectivity radius, and the conclusion of the lemma makes sense.

	With all our tools in hand, we are ready to prove Theorem \ref{thm: generic echolocation}.
	
	\subsection{Proof of Theorem \ref{thm: generic echolocation}}
	
	We may express $\mathcal G$ as the countable union
	\[
	\mathcal G = \bigcup_{j = 0}^\infty \mathcal G_j
	\]
	where $\mathcal G_j$ is the set of those metrics whose injectivity radius is larger than $2^{-j}$. It suffices then to show Theorem \ref{thm: generic echolocation} holds for $\mathcal G_j$ for each $j$, which follows from the case $j = 0$ by rescaling.
	
	Let $\Delta$ denote the diagonal in $M \times M$, and let $\mathcal H$ denote the set of all triples $(g,x,y) \in \mathcal G_0 \times (M \times M \setminus \Delta)$ for which the counting functions $N_x$ and $N_y$ are identical. If we can satisfy the hypotheses of Theorem \ref{thm: meager}, it yields a residual set of metrics in $\mathcal G_0$, for which
	\[
	N_x \neq N_y \text{ whenever } x \neq y \text{ for each } x,y \in M,
	\]
	as desired. So, we must construct $(2n+1)$-dimensional slices across $\mathcal H$ at each point in $(g_0, x_0, y_0) \in \mathcal G_0 \times M \times M$, $x_0 \neq y_0$.
	
	Take a sequence of times $t_0,t_1,\ldots$ such that, for all $(g,x,y)$ in a neighborhood of $(g_0,x_0,y_0)$,
	\[
	t_0 < \frac 12 \min(d_g(x,y), 1) \qquad \text{ and } \qquad t_{k} = \frac 12 t_{k-1} \text{ for $k = 1,2,\ldots$}.
	\]
	Then, take $\pi$ in Definition \ref{def: slice} to be the map from this neighborhood to $\R^{2n+1}$ given by
	\[
	\pi(g,x,y) = \begin{bmatrix}
		\varphi(g,t_0,x) - \varphi(g,t_0,y) \\
		\vdots \\
		\varphi(g,t_{2n},x) - \varphi(g,t_{2n},y)
	\end{bmatrix}.
	\]
	Note, $(g,x,y) \in \mathcal H$ implies $\pi(g,x,y) = 0$ by Proposition \ref{prop: cosine transform}. 
	
	We construct $\rho$ of Definition \ref{def: slice} by taking smooth functions $h_0, \ldots, h_{2n}$ on $M$ and setting
	\[
	\rho(s) = e^{s \cdot h} g_0,
	\]
	where here we have written
	\[
	s \cdot h = s_0 h_0 + \cdots + s_{2n} h_{2n}
	\]
	as shorthand. Shrinking the domain of $\rho$ as necessary to a small enough neighborhood of the origin, we set
	\[
	\Phi(s,x,y) = \pi(\rho(s),x,y) = \begin{bmatrix}
		\varphi(e^{s \cdot h} g_0,t_0,x) - \varphi(e^{s \cdot h} g_0,t_0,y) \\
		\vdots \\
		\varphi(e^{s \cdot h} g_0,t_{2n},x) - \varphi(e^{s \cdot h} g_0,t_{2n},y)
	\end{bmatrix}.
	\]
	Lemma \ref{lem: C^1 2} guarantees $\Phi$ is $C^1$. Finally, to verify $\Phi$ is a slice, we must select $h_0,\ldots,h_{2n}$ cleverly enough so that
	\[
	d_s \Phi(g_0, x_0, y_0) = \begin{bmatrix}
		D_{h_j} \varphi(g_0,t_k,x_0) - D_{h_j} \varphi(g_0,t_k,y_0)
	\end{bmatrix}_{k,j = 0}^{2n}
	\]
	is nonsingular. Lemma \ref{lem: h construction} allows us to select $2n + 1$ smooth functions $h_0, h_1,\ldots, h_{2n}$ for which
	\[
	D_{h_k} \varphi(g,2^{-k},x) = 1 \quad \text{ and } \quad \supp h_k \subset \left\{z \in M : \frac 1 3 2^{-k} < d_g(x,z) < \frac 2 3 2^{-k} \right\}.
	\]
	By construction and Proposition \ref{prop: support}, we also have
	\[
	D_{h_k} \varphi(g,2^{-k},y) = 0 \qquad \text{ for } k \in \{k_0,\ldots, k_0 + 2n\}.
	\]
	Hence,
	\[
	d_s \Phi(g_0, x_0, y_0) = \begin{bmatrix}
		1 & * & * & \cdots & * \\
		0 & 1 & * & \cdots & * \\
		0 & 0 & 1 & \cdots & * \\
		\vdots & \vdots & \vdots & \ddots & \vdots \\
		0 & 0 & 0 & \cdots & 1
	\end{bmatrix},
	\]
	which is nonsingular as desired.
	
	To summarize, we have produced a $(2n + 1)$-dimensional slice across $\mathcal H$ at each point $\{ (g_0, x_0, y_0) \in \mathcal G_0 \times M \times M : x_0 \neq y_0\}$. Theorem \ref{thm: meager} then yields a residual class of metrics in $\mathcal G_0$ for which $N_x = N_y$ only when $x = y$. Theorem \ref{thm: generic echolocation} follows.

	\section{Proof of Lemma \ref{lem: h construction}} \label{sec: proof lemma}
	
	At first glance, it seems absurd that the conclusion of Lemma \ref{lem: h construction} could possibly fail to hold. We have the freedom to take any perturbation we like from the infinite-dimensional space of options we have for $h$ as long as it has the required support, and so surely we can hear at least one of these perturbations at $x$, if not most.
	
	However, when we consider the $n = 1$ case, we run into a conceptual obstruction. There are no local geometric invariants on a one-dimensional Riemannian manifold, and in particular, we expect no perturbation to be audible for small times. This means that, at least when $n = 1$, there is a complete cancellation between the contribution of the two terms of $[D_h, \Delta_g]$, which forces $D_h \varphi(g,t,x) = 0$. We must show that this can be avoided for $n \geq 2$.  This is consistent with the fact that sharp Huygens' principle always holds in dimension one, but gets more complicated in higher dimensions. See, e.g., \cite{gunther}.
	
	To this end, we will choose a convenient perturbation $h$ and then carefully estimate $D_h \varphi(g,t,x)$ to ensure it is nonzero. For this, we select $h$ to be a rapidly oscillating function with the desired support. In effect, this allows us to use oscillatory testing to extract main- and remainder-term asymptotics for $D_h \varphi(g,t,x)$ as the frequency of $h$ tends to infinity.
	
	The main result of this section is:
	
	\begin{lemma}\label{lem: h asymptotics}
		Fix geodesic normal coordinates $(z_1,\ldots,z_n) \in \R^n$ about $x$, and take
		\[
		h(z) = a(z) \cos(\lambda z_1 + \theta)
		\]
		where $\theta \in \R$ and $a$ is a smooth function supported in $\{z : |z - t/2| < t/6\}$ with $a(t/2, 0,\ldots,0) = 1$. Then,
		\[
		D_h \varphi(g,t,x) = (2\pi)^{-\frac{n-1}{2}} \frac{n - 1}{16} \lambda^{\frac{n+1}{2}} t^{-\frac{n-1}{2}} \cos(\pi(n-3)/4 + t/2 + \theta) + O(\lambda^{\frac{n-1}{2}})
		\]
		as $\lambda \to \infty$.
	\end{lemma}
	
	To see how Lemma \ref{lem: h construction} follows, we first take $\theta = -\pi(n-3)/4 - t/2$ to ensure the main term is nonzero. Then, we ensure the left side is nonzero by selecting a large enough $\lambda$ so that the main term is strictly greater than the remainder.
	
	\subsection{A quick rephrasing}
	
	As an operator taking functions to functions (rather than densities to functions, half-densities to half-densities, etc.), the kernel of the sine wave operator satisfies
	\[
	\frac{\sin(r\sqrt{-\Delta_g})}{\sqrt{-\Delta_g}}(x,z) = \frac{\sin(r\sqrt{-\Delta_g})}{\sqrt{-\Delta_g}}(z,x) |g(z)|^{1/2}  |g(x)|^{-1/2}.
	\]
	Hence by \eqref{eq: D_h phi}, 
	\begin{multline*}
		D_h \varphi(g,t,x) = |g(x)|^{-1} \int_0^t \int_M \frac{\sin((t - s) \sqrt{-\Delta_g})}{\sqrt{-\Delta_g}}(z,x) \\
		\cdot [D_h, \Delta_g] \cos(s \sqrt{-\Delta_g})(z,x) |g(z)|^{1/2} \, dz \, ds,
	\end{multline*}
	where here $[D_h, \Delta_g]$ acts in the $z$ variable of $\cos(s \sqrt{-\Delta_g})(z,x)$. Using geodesic normal coordinates $z = (z_1,\ldots,z_n)$ about $x = (0,\ldots,0)$ as in the statement of the lemma, we obtain
	\begin{multline*}
		D_h \varphi(g,t,x) = \int_0^t \int_{\R^n} \frac{\sin((t - s) \sqrt{-\Delta_g})}{\sqrt{-\Delta_g}}(z,0) \\
		\cdot [D_h, \Delta_g] \cos(s \sqrt{-\Delta_g})(z,0) |g(z)|^{1/2} \, dz \, ds.
	\end{multline*}
	
	At this point, it will be convenient to introduce some shorthand. We let
	\[
	S(s,z) := \frac{\sin(s\sqrt{-\Delta_g})}{\sqrt{-\Delta_g}}(z,0)
	\]
	and
	\[
	C(s,z) := \cos(s\sqrt{-\Delta_g})(z,0).
	\]
	We consider the contribution of the second term of $[D_h, \Delta_g]$ in Lemma \ref{lem: commutator}. In particular, by a distributional integration by parts, we have
	\begin{multline*}
		\left(\frac n 2 - 1 \right) \int_{\R^n} S(t-s,z) \nabla_g h(z) \cdot \nabla_g C(s,z) |g(z)|^{1/2} \, dz \\
		= - \left(\frac n 2 - 1 \right) \int_{\R^n} h(z) \nabla_g S(t-s,z) \cdot \nabla_g C(s,z) |g(z)|^{1/2} \, dz \\
		- \left(\frac n 2 - 1 \right) \int_{\R^n} h(z) S(t-s,z) \cdot \Delta_g C(s,z) |g(z)|^{1/2} \, dz.
	\end{multline*}
	Hence in total,
	\begin{multline*}
		D_h \varphi(g,t,x) = \\
		- \int_{\R^n} \left( \int_0^t \frac n 2 S(t-s,z) \Delta_g C(s,z) + \Big(\frac n 2 - 1 \Big) \nabla_g S(t-s,z) \cdot \nabla_g C(s,z) \, ds \right) \\
		\cdot h(z) |g(z)|^{1/2} \, dz \, ds
	\end{multline*}
	We recall
	\[
	h(z) = a(z) \cos(\lambda z_1 + \theta) = \Re \left(a(z) e^{-i(\lambda z_1 + \theta)} \right) 
	\]
	and note that, since both the left side and the integral in parentheses are real-valued, we may write
	\[
	D_h \varphi(g,t,x) = \Re (e^{-i\theta} \widehat u(0,\lambda e_1))
	\]
	where $u$ is the distribution on $\R^{1+n}$ given by
	\begin{multline*}
		u(s,z) = - a(z) |g(z)|^{1/2} \mathbf 1_{[0,t]}(s) \\
		\cdot \left(\frac n 2 S(t-s,z) \Delta_g C(s,z) + \Big(\frac n 2 - 1 \Big) \nabla_g S(t-s,z) \cdot \nabla_g C(s,z) \right)
	\end{multline*}
	and $\widehat u$ denotes the Fourier transform on $\R^{1+n}$.
	
	We have an opportunity now to smooth out the indicator function $\mathbf 1_{[0,t]}$. To do this, consider a smooth cutoff $\gamma$ with $\gamma(s) = 1$ for $|s - t/2| \leq t/6$ and $\gamma(s) = 0$ for $|s - t/2| \geq t/4$. We claim that $(1 - \gamma(s))u(s,z)$ is smooth in both variables, and hence
	\[
	\widehat u(0,\lambda e_1) = \widehat{\gamma u}(0,\lambda e_1) + O(\lambda^{-\infty}).
	\]
	To see this, we only need to note that $S(t-s,z)$ is smooth for $|t - s| \neq |z|$ and $C(s,z)$ is smooth for $|s| \neq |z|$. Now if $(s,z)$ is in the support of $(1 - \gamma(s))u(s,z)$, then
	\begin{enumerate}
		\item $|s - t/2| \geq t/6$, and
		\item $z \in \supp a$, and hence $|z - \frac t 2 e_1| < t/6$, and hence $||z| - t/2| < t/6$.
	\end{enumerate}
	Hence, if $(s,z)$ lies in the support of $(1 - \gamma(s))a(z)$ and $0 \leq s \leq t$, then $(1 - \gamma(s))u(s,t)$ is smooth on a neighborhood of $(s,z)$.
	
	We now set
	\begin{multline*}
		v(s,z) = \gamma(s) u(s,z) = - a(z) |g(z)|^{1/2} \gamma(s) \\
		\cdot \left(\frac n 2 S(t-s,z) \Delta_g C(s,z) + \Big(\frac n 2 - 1 \Big) \nabla_g S(t-s,z) \cdot \nabla_g C(s,z) \right)
	\end{multline*}
	and note
	\begin{equation}\label{eq: D_h phi real part}
		D_h \varphi(g,t,x) = \Re( e^{-i\theta} \widehat v(0,\lambda e_1)) + O(\lambda^{-\infty}).
	\end{equation}
	
	\subsection{Hadamard's parametrix} A \emph{symbol} of order $m$ on $\R^n \times \R^N$ is a smooth function $a$ on $\R^n \times \R^N$ satisfying bounds
	\[
	|\partial_\theta^\beta \partial_x^\alpha a(x,\theta)| \leq C_{\alpha,\beta}(1 + |\theta|)^{m - |\beta|}
	\]
	for multiindices $\alpha$ and $\beta$ (see e.g. \cite{SFIO, Hang, DuistermaatFIOs, HIII}). The set of such symbols is denoted $S^m(\R^n \times \R^N)$. We typically consider $x \in \R^n$ the spatial variables and $\theta \in \R^N$ the frequency variables.

	Hadamard's parametrix allows us to write the distribution kernels of the sine and cosine wave operators as oscillatory integrals. We will rephrase the characterization of Hadamard's parametrix as it appears in \cite[Chapter 2]{Hang}. First, consider $z = (z_1,\ldots,z_n)$ in geodesic normal coordinates about $x$ as in Lemma \ref{lem: h asymptotics}. Then, there exist symbols $b_\pm(s,z,\xi)$ in $S^0(\R^{1 + n} \times \R^n)$ with
	\[
	b_{\pm}(s,z,\xi) - \frac 12 \in S^{-2}(\R^{1 + n} \times \R^n)
	\]
	and
	\begin{equation}\label{eq: cosine parametrix}
		C(s,z) = (2\pi)^{-n} |g(z)|^{-1/4} \sum_\pm \int_{\R^n} e^{i(\langle z, \xi \rangle \pm s|\xi|)} b_\pm(s,z,\xi) \, d\xi + R_1(s,z)
	\end{equation}
	where $R_1$ is a smooth discrepancy. Furthermore, at the expense of absorbing a smooth error into $R_1$, we may take $b$ to be supported for $|\xi| \geq 1$. Similarly, there exist symbols $c_\pm(s,z,\xi)$ of order $-1$ with
	\[
	c_\pm(s,z,\xi) - \frac{1}{\pm 2 i|\xi|} \in S^{-3}(\R^{1 + n} \times \R^n)
	\]
	and support in $|\xi| \geq 1$ such that
	\begin{equation}\label{eq: sine parametrix}
		S(s,z) = (2\pi)^{-n} |g(z)|^{-1/4} \sum_\pm \int_{\R^n} e^{i(\langle z,\xi \rangle \pm s|\xi|)} c_{\pm}(s,z,\xi) \, d\xi + R_2(s,z),
	\end{equation}
	where $R_2$ is again a smooth discrepancy.
	
	\subsection{Reduction to an oscillatory integral} To make our task easier, we start by breaking $v$ into two terms,
	\[
	v = v_1 + v_2
	\]
	where
	\[
	v_1(s,z) = - a(z) |g(z)|^{1/2} \gamma(s) \frac n 2 S(t-s,z) \Delta_g C(s,z)
	\]
	and 
	\[
	v_2(s,z) = - a(z) |g(z)|^{1/2} \gamma(s) \Big(\frac n 2 - 1 \Big) \nabla_g S(t-s,z) \cdot \nabla_g C(s,z).
	\]
	
	We first examine $v_1$. Note,
	\begin{align*}
		\Delta_g C(s,z) &= \partial_s^2 C(s,z) \\
		&= (2\pi)^{-n} |g(z)|^{-1/4} \sum_\pm \int_{\R^n} e^{i(\langle z, \xi \rangle \pm s|\xi|)} (-|\xi|^2) b_\pm(s,z,\xi) \, d\xi + \partial_s^2 R_1(s,z)
	\end{align*}
	where in the second line the lower-order parts of the symbol $b_\pm$ have changed. One quickly sees this by showing
	\[
	e^{-i(\langle z, \xi \rangle \pm s|\xi|)} \partial_s^2 \left( e^{i(\langle z, \xi \rangle \pm s|\xi|)} b_\pm(s,z,\xi) \right) + |\xi|^2 b_\pm(s,z,\xi) \in S^1(\R^{1 + n} \times \R^n).
	\]
	Then,
	\begin{multline*}
		v_1(s,z) \\
		= (2\pi)^{-2n} \frac n 2 \sum_{\pm,\pm'} \iint e^{i(\langle z, \eta + \xi \rangle \pm' (t-s)|\eta| \pm s|\xi|)} \gamma(s) a(z) c_{\pm'}(t-s,z,\eta) b_\pm(s,z,\xi) |\xi|^2 \, d\eta \, d\xi \\
		- (2\pi)^{-n} \frac{n}{2} |g(z)|^{1/4} \sum_{\pm'} \int e^{i(\langle z, \eta \rangle \pm'(t-s)|\eta|)} \gamma(s) a(z) c_{\pm'}(t-s,z,\eta) \partial_s^2 R_1(s,z) \, d\eta \\
		+ (2\pi)^{-n} \frac{n}{2} |g(z)|^{1/4} \sum_{\pm} \int e^{i(\langle z, \xi \rangle \pm s|\xi|)} |\xi|^2 \gamma(s) a(z) b_{\pm}(s,z,\eta) R_2(t-s,z) \, d\xi \\
		- \frac n 2 a(z) |g(z)|^{1/2} \gamma(s) R_2(t-s,z) \partial_s^2 R_1(s,z),
	\end{multline*}
	where $b_{\pm}$ and $c_{\pm'}$ are as they appear in \eqref{eq: cosine parametrix} and \eqref{eq: sine parametrix} except perhaps up to lower-order symbols. Now, the fourth term on the right is smooth in all variables, so it contributes a negligible term to $\widehat{v}(0,\lambda e_1)$. We can see that the second and third terms also yield a rapidly-vanishing contribution by multiplying by integrating in $s$, and then using integration by parts in $s$ to obtain an arbitrarily smooth function of $z$. Hence, we are left with having to contend with the first term.
	
	We may repeat the process for $v_2$, except we fix indices $i,j \in \{1,\ldots,n\}$ and replace $S$ with $\partial_{z_i} S$ and $\Delta_g C$ with $\partial_{z_j} C$. Tracing through the steps above, we obtain a main term of
	\begin{multline*}
		\left(\frac n 2 - 1 \right) \sum_{i,j} g^{ij}(z) \sum_{i,j} \sum_{\pm,\pm'} \iint e^{i(\langle z, \eta + \xi \rangle \pm' (t-s)|\eta| \pm s|\xi|)} \\
		\cdot a(z) c_{\pm'}(t-s,z,\eta) b_\pm(s,z,\xi) \eta_i \xi_j \, d\eta \, d\xi
	\end{multline*}
	Combining the main terms of $v_1$ and $v_2$ and taking a Fourier transform yields:
	
	\begin{lemma} \label{lem: pre-cut fourier integral}
		We have
		\begin{multline*}
			\widehat v(0,\lambda e_1) = (2\pi)^{-2n} \sum_{\pm,\pm'} \iiiint e^{i(\langle z, \eta + \xi - \lambda e_1 \rangle \pm' (t-s)|\eta| \pm s|\xi|)} \\
			\cdot \left( \frac n 2 |\xi|^2 + \left(\frac n 2 - 1\right)\langle \eta, \xi \rangle_{g(z)} \right) \gamma(s) a(z) c_{\pm'}(t-s,z,\eta) b_\pm(s,z,\xi) \, ds \, dz \, d\xi \, d\eta \\
			+ O(\lambda^{-\infty})
		\end{multline*}
		where $b_\pm$ and $c_{\pm'}$ are as they appear in \eqref{eq: cosine parametrix} and \eqref{eq: sine parametrix}, respectively, up to addition by a lower-order symbol.
	\end{lemma}
	
	\subsection{Preparation for the method of stationary phase}
	
	We perform a change of variables $(\eta,\xi) \mapsto (\lambda \eta, \lambda \xi)$ in the integral in the above lemma and obtain an oscillatory integral
	\[
	(2\pi)^{-2n} \lambda^{2n} \sum_{\pm,\pm'} \iiiint e^{i\lambda \phi_{\pm,\pm'}(s,z,\eta,\xi)} \alpha_{\pm,\pm'}(s,z,\lambda \eta, \lambda \xi) \, d\xi \, d\eta \, dz \, ds
	\]
	with phase function
	\[
	\phi_{\pm,\pm'}(s,z,\eta,\xi) = \langle z, \eta + \xi - e_1 \rangle \pm' (t - s)|\eta| \pm s|\xi|
	\]
	and where $\alpha_{\pm,\pm'} \in S^1(\R^{1 + n} \times \R^{2n})$ with
	\begin{multline}\label{eq: amplitude def}
		\alpha_{\pm,\pm'}(s,z,\eta,\xi) \\
		= \left( \frac n 2 |\xi|^2 + \left(\frac n 2 - 1\right)\langle \eta, \xi \rangle_{g(z)} \right) \gamma(s) a(z) c_{\pm'}(t-s,z,\eta) b_\pm(s,z,\xi).
	\end{multline}
	We now make some cuts. 
	
	First, we let $\beta$ be a smooth function with $\beta(\xi) = 1$ for $|\xi| \leq 1/8$ and $\beta(\xi) = 0$ for $|\xi| \geq 1/4$. We cut the integral into  $\beta(\xi)\beta(\eta)$ and $1 - \beta(\xi)\beta(\eta)$ parts. For the former part, we integrate by parts in $z$ with the operator
	\[
	L = \frac{\eta + \xi - e_1}{i\lambda |\eta + \xi - e_1|^2} \cdot \nabla_z
	\]
	plenty of times to obtain something which vanishes to arbitrary order as $\lambda \to \infty$. We are then left with the latter part. We now have an amplitude supported in the set where $|\xi| \geq 1/8$ or $|\eta| \geq 1/8$. To save our notation from becoming cluttered, we will simply absorb the $1 - \beta(\xi)\beta(\eta)$ into the amplitude $\alpha_{\pm,\pm'}$ and free up the letter $\beta$ to be some other cutoff.
	
	Next, we argue that we obtain rapid decay for the terms where the signs $\pm$ and $\pm'$ disagree. We integrate by parts in $s$ with the operator
	\[
	L = \frac{1}{i\lambda(|\xi| + |\eta|)} \partial_s.
	\]
	Note that the denominator of the operator is bounded away from $0$ since $|\xi|$ and $|\eta|$ cannot both be small (less than $1/8$) simultaneously on the support of the integrand. Doing so plenty of times yields an $L^1$ integrand and plenty of decay in $\lambda$. Hence from here on out, we will only consider the case $\pm' = \pm$.
	
	Next, we exclude a neighborhood of the axes $\xi = 0$ or $\eta = 0$. For this, let $\beta(\eta) = 1$ for $|\eta| \leq 1/20$ and $\beta(\eta) = 0$ for $|\eta| \geq 1/10$. We cut the integrand into $\beta(\eta)$ and $1 - \beta(\eta)$ parts. On the support former part, we necessarily have $|\xi| \geq 1/8$, and hence $||\xi| - |\eta|| \geq 1/40$. Integrating by parts plenty of times using operator
	\[
	L = \frac{1}{i\lambda||\xi| - |\eta||} \partial_s
	\]
	then yields an amplitude which is compactly-supported in $\eta$ and $L^1$ in $\xi$, and also as much decay in $\lambda$ as desired. Hence, we are left with the $1 - \beta(\eta)$ part of the integral. We repeat the argument similarly to ensure the amplitude is also supported for $|\xi| \geq 1/20$.
	
	Next, we eliminate the term $\pm = \pm' = +$. This is done by noting that the phase function has no critical points on the support of the amplitude and integrating by parts with the operator
	\[
	L = \frac{1}{i\lambda|\nabla \phi|^2} \nabla \phi \cdot \nabla
	\]
	where the gradients are in all four variables $s,z,\xi,$ and $\eta$. It suffices to consider the one term with $\pm = \pm' = -$.
	
	Finally, we compactify the support of the amplitude. In particular, take $\beta(\xi, \eta) = 1$ for $|\xi| \leq 1$ and $|\eta| \leq 1$ and $\beta(\xi,\eta) = 0$ for $|\xi| \geq 2$ or $|\eta| \geq 2$. We make one last cut into $\beta(\xi,\eta)$ and $1 - \beta(\xi,\eta)$ parts and claim the latter is negligible. Note that if $(s,z,\xi,\eta)$ is a critical point of the phase function, we necessarily have
	\[
	|\xi| = |\eta|, \quad \eta + \xi = e_1, \quad \text{ and } \quad z = s\xi/|\xi| = (t - s)\eta/|\eta|,
	\]
	which has a unique critical point at
	\[
	(s,z,\eta,\xi) = \left( \frac t2, \frac t2 e_1, \frac 12 e_1, \frac12 e_1 \right),
	\]
	which is certainly not included in the support of this cut of the amplitude. We then integrate by parts using the same operator as above.
	
	Before recording our progress, we make one final observation. From \eqref{eq: amplitude def}, we have
	\begin{multline*}
		\alpha_{-,-}(s,z,\lambda \eta, \lambda \xi) \\
		= \lambda^2 \left( \frac n 2 |\xi|^2 + \left(\frac n 2 - 1\right)\langle \eta, \xi \rangle_{g(z)} \right) \gamma(s) a(z) c_{-}(t-s,z,\lambda \eta) b_-(s,z, \lambda \xi).
	\end{multline*}
	Since we are in geodesic normal coordinates, we have
	\[
	g(z) = \begin{bmatrix}
		1 & 0 \\
		0 & *
	\end{bmatrix} \qquad \text{ at } z = \frac t 2 e_1,
	\]
	and hence
	\[
	\langle \eta, \xi \rangle_{g(z)} = \eta_1 \xi_1 \qquad \text{ at } z = \frac t 2 e_1.
	\]
	It follows that at the critical point,
	\[
	\alpha_{-,-}\left(\frac t 2, \frac t2 e_1 ,\lambda \frac 1 2 e_1, \lambda \frac 1 2 e_1 \right) = \frac{i}{8}(n-1)\lambda + O(1).
	\]
	To summarize, we have shown:
	
	\begin{lemma} \label{lem: pre-stationary phase} The integral in Lemma \ref{lem: pre-cut fourier integral} is equal to
		\[
		(2\pi)^{-2n} \lambda^{2n} \iiiint e^{i\lambda \phi(s,z,\eta,\xi)} \alpha(\lambda; s,z,\eta,\xi) \, dz \, d\xi \, d\eta \, ds + O(\lambda^{-\infty})
		\]
		with phase function
		\[
		\phi(s,z,\eta,\xi) = \langle z, \eta + \xi - e_1 \rangle - (t - s)|\eta| - s|\xi|
		\]
		and where the amplitude $\alpha$ is a symbol of order $1$ in $\lambda$ with
		\[
		\alpha\left(\lambda; \frac t 2, \frac t 2 e_1, \frac 1 2 e_1, \frac 1 2 e_1 \right) = \frac{i}{8}(n-1)\lambda + O(1),
		\]
		and with support on $(s,z,\eta,\xi)$ satisfying
		\[
		\left|\frac t 2 - s\right| \leq \frac t 6, \quad \left|z - \frac t 2 e_1 \right| \leq \frac t 6, \quad \frac{1}{20} \leq |\xi| \leq 2, \quad \text{ and } \quad \frac{1}{20} \leq |\eta| \leq 2.
		\]
	\end{lemma}
	
	\subsection{The application of stationary phase}
	
	Finally, we apply the method of stationary phase to obtain asymptotics for the expression in Lemma \ref{lem: pre-stationary phase}. At the unique critical point, the Hessian of the phase function reads as
	\[
	\nabla^2 \phi = \begin{bmatrix}
		0 & 0 & e_1^t & -e_1^t \\
		0 & 0 & I & I \\
		e_1 & I & t(e_1e_1^t - I) & 0 \\
		-e_1 & I & 0 & t(e_1e_1^t - I)
	\end{bmatrix}.
	\]
	We compute the determinant and signature of the matrix by using some row and column operations. By taking the $(2n+2)$-th row and moving it up to occupy row $n + 2$, and similarly for the columns, we find the Hessian matrix above is conjugate to
	\[
	\left[
	\begin{array}{cc|c|cc|c|c}
		&  &  & 1 & -1 &  &  \\
		&  &  & 1 & 1 &  &  \\ \hline 
		&  &  &  &  & I_{n-1} & I_{n-1} \\ \hline
		1 & 1 &  &  &  &  &  \\
		-1 & 1 &  &  &  &  &  \\ \hline
		&  & I_{n-1} &  &  & -tI_{n-1} &  \\ \hline
		&  & I_{n-1} &  &  &  & -tI_{n-1}
	\end{array}
	\right].
	\]
	After further conjugating by elementary determinant-$1$ row operations, we obtain
	\[
	\left[
	\begin{array}{cc|c|cc|c|c}
		&  &  & -1 & 1 &  &  \\
		&  &  & 1 & 1 &  &  \\ \hline 
		&  &  &  &  & I_{n-1} &  \\ \hline
		-1 & 1 &  &  &  &  &  \\
		1 & 1 &  &  &  &  &  \\ \hline
		&  & I_{n-1} &  &  &  &   \\ \hline
		&  &  &  &  &  & -tI_{n-1}
	\end{array}
	\right]
	\]
	from which it is relatively straightforward to find both the determinant and the signature $\sigma$. In particular, we have
	\[
	|\det \nabla^2 \phi| = 4t^{n-1} \quad \text{ and } \quad \sigma = -n+1.
	\]
	We use the method of stationary phase (e.g. \cite[Theorem 7.7.5]{HI} or \cite[Proposition 1.2.4]{DuistermaatFIOs}) in all $3n + 1$ variables to obtain
	\[
	\widehat v(0,\lambda e_1) = (2\pi)^{-\frac{n-1}{2}} \frac{1}{16} \lambda^{\frac{n + 1}{2}} t^{-\frac{n-1}{2}} e^{-i \left(\frac{\pi}{4} (n-3) + \frac t 2 \right)} + O(\lambda^{\frac{n-1}{2}}).
	\]
	(The $t/2$ in the exponential is due to the phase function taking the value $-t/2$ at the critical point.)
	At last, we obtain Lemma \ref{lem: h asymptotics} after feeding this asymptotics into \eqref{eq: D_h phi real part}.
	
	\section{Further directions} \label{sec: further directions} 
	
	The purpose of this paper is to introduce the echolocation question (Question \ref{echolocation question}) and to resolve it generically. Specific special cases should be addressed next.
	
	\subsection{Specific examples} The examples provided in Section \ref{sec: examples} are only the very start of an exploration of Question \ref{echolocation question}. We list a few more here which we believe are tractable. 
	
	The first is the example of the rectangle with Dirichlet boundary conditions, which was explored in the one- and two-dimensional cases in Section \ref{sec: examples}. {What happens  in higher dimensions when multiplicities are allowed?}
	
	There are other obvious choices of Euclidean domains that should be explored, such as ellipses and triangles, both with Dirichlet (or Neumann) boundary conditions.
	
	Zoll manifolds may provide another source of interesting examples. The paper of Zelditch \cite{zoll} provides fine estimates of the local (pointwise) Weyl counting function $N_x(\lambda)$ \eqref{pointwise counting}. It would be interesting to see if these estimates can be used to establish ``echolocation'' on such a manifold. We plan to work on this problem in the near future.
	
	\subsection{Hearing the shape of a submanifold through its Kuznecov sum.} Let $H$ be a compact boundaryless smooth connected submanifold embedded in $M$. Consider the Kuznecov sum
	\begin{equation}\label{Kuz}
		N_H(\lambda)=\sum_{\lambda_j\le\lambda}\left|\int_H e_j(x)\,dV_H(x)\right|^2.
	\end{equation}
	The local Weyl counting function $N_x(\lambda)$ \eqref{pointwise counting} can be seen as a special case of \eqref{Kuz} when $H$ is a single point $\{x\}.$ General Weyl-type asymptotics for \eqref{Kuz} were established by Zelditch\cite{z92}.  
	
	Similar to the Weyl counting function, one can also ``hear'' the volume of $H$ by looking at the top term of the Kuznecov sum. Furthermore, a recent work of the authors \cite{2term} shows that one can obtain a significant amount of geometric knowledge of $H$ from just the first two terms of \eqref{Kuz}. It is then natural to ask the following question.
	
	\begin{question}[Hearing the shape of a submanifold]\label{Kuznecov question}
		Can one determine the embedding $H\hookrightarrow M$ up to isometry, given the complete knowledge of $M$ and the sum \eqref{Kuz}?
	\end{question} 
	
	Question \ref{Kuznecov question} is more complicated than Question \ref{echolocation question}. To determine the embedding, one needs to determine both the shape and the location of $H$ in $M$. Even though we do not have a definite answer to this question, we believe that, at the very least, some special kind of submanifolds could be uniquely determined by their Kuznecov sum \eqref{Kuz}. 
	
	There is a natural Euclidean counterpart to this question. Suppose that $S$ is a smooth compact manifold embedded in $\mathbb R^d$. Let $d\sigma$ be the induced Lebesgue measure on $S$. Consider the quantity
	\begin{equation}\label{Euclidean}
		N_S(\lambda)=\int_{|\xi|\le\lambda}|\widehat{d\sigma}|^2(\xi)\,d\xi.
	\end{equation}
	\begin{question}[Euclidean 
		version]\label{Euclidean question}
		Can one determine the shape of $S$ given the complete knowledge of \eqref{Euclidean}? 
	\end{question}
	
	To further expand on Question \ref{Kuznecov question}, we may also consider the case when $H$ is disconnected. The simplest such case is when $H$ is the set of two distinct points $\{x,y\}$ on $M$. Then the corresponding Kuznecov sum is
	
	\begin{equation}\label{two point}
		N_{x,y}(\lambda)=\sum_{\lambda_j\le\lambda}|e_j(x)+e_j(y)|^2.
	\end{equation}
	\begin{question}[Simultaneous echolocation]\label{two point question}
		Can one determine the locations of both $x$ and $y$ up to isometry, given the complete knowledge of $M$ and the sum \eqref{two point}? 
	\end{question}
	
	{\subsection{Echolocation on graphs.}
		It is known that there are many examples of cospectral graphs (see, e.g., \cite{butler}). These are pairs of non-isomorphic graphs that have the same multiset of Laplacian eigenvalues. It is interesting to consider echolocation on a given graph. More precisely, let $G$ be a graph with $n$ nodes $\{v_i\}_{i=0}^{n-1}$ and normalized Laplacian matrix $\mathcal L$. Suppose that $\mathcal L$ has eigenvalues $0=\lambda_0\le\lambda_1\le\cdots\le\lambda_{n-1}$ and a corresponding orthonormal eigenbasis $\{e_j\}_{j=0}^{n-1}$. 
		Consider
		\begin{equation}
			\label{graph}
			N_{v_i}(\lambda)=\sum_{\lambda_j\le\lambda}|e_{ji}|^2.
		\end{equation}
		In graph theory, two vertices on a graph are said to be similar if there is an automorphism of the graph that maps one to the other. One may ask, in a given graph, can non-similar vertices share the same \eqref{graph}? This problem has been extensively studied, especially if we replace $ \mathcal L $ with the adjacent matrix $ A $. Two points that share the same $N_{v_i}(\lambda)$ with respect to the adjacent matrix $A$ are called {\it cospectral vertices}. This concept dates back to Schwenk \cite{schwenk73}, where a tree of 9 nodes containing a pair of non-similar cospectral vertices is given. It is now well-known that many graphs with non-similar cospectral vertices exist. The same concept extends to the context of normalized Laplacian $\mathcal L$, see, e.g., \cite{godsil93,coutinho23}. These two versions of cospectral vertices coincide for regular graphs. It would be interesting to find the minimal (regular) graphs
		on which echolocation fails and characterize them. We wish to explore this problem in the near future.}
	
	\subsection{Echolocation with incomplete knowledge} As discussed in remark
	\ref{incomplete remark}, for certain generic cases, it is possible to hear the Weyl counting function at a single point on the manifold. Thus, in some special cases, one should be able to hear the shape of the manifold and simultaneously infer one's location on the manifold.
	
	To illustrate this, consider a string $M=[0,a]$ with Dirichlet boundary conditions. In section \ref{sec: examples}, we showed that ``echolocation'' holds on $M$. The first eigenfunction on $M$ is of the form
	\[ e_1(x) = \sqrt{\frac{2}{a}} \sin(\pi x/a).\]
	Then for any $x\in(0,\pi),$ $e_1(x)\neq 0.$ So one can recover the first non-zero eigenvalue $\lambda_1=\pi/a$ from $N_x(\lambda)$ at any $x$ in the interior of the string. The length of the string is then determined by $a=\pi/\lambda_1$. It would be interesting to find more examples like this.

	For a manifold, { we generally cannot} determine its shape from $N_x(\lambda)$ at a point $x$ alone. Nevertheless, animals in the real world with functioning echolocation systems usually do not know the geometry of their environment in advance. They often have two ears with which they can perceive not only the volume and frequency of echoes but also the directions from which the echoes come. It is a somewhat vague but exciting question whether one can hear the shape of a manifold and one's location in it at the same time if one can hear the directional vectors of all the echoes in addition to $N_x(\lambda)$.
	
	\subsection{Echolocation with finitely many audible quantities}

	In the proof of Theorem \ref{thm: generic echolocation}, we need to use infinitely many audible quantities, namely $\varphi(g,t_k,x)$ defined in \eqref{eq: renormalized wave} for $k = 0,1,2,\ldots$. We show that for a generic class of smooth metrics, there is an injective map $M \to \R^\N$ whose coordinates are all audible quantities. This may be more information than needed, so we can ask a stronger version of Question \ref{echolocation question}.
	
	\begin{question}
		Let $M$ be a compact Riemannian manifold with or without boundary and let $\operatorname{Isom}(M)\backslash M$ denote the quotient of $M$ by the action of its isometry group $\operatorname{Isom}(M)$. What is the smallest integer $k$ such that there exist smooth audible quantities $\varphi_1, \ldots, \varphi_k$ for which the map
		\[
		x \mapsto (\varphi_1(x), \ldots, \varphi_k(x))
		\]
		induces a continuous embedding $\operatorname{Isom}(M)\backslash M \to \R^k$?
	\end{question}
	
	In analogy with the numerology of other generic embedding results, such as the weak Whitney embedding theorem and Takens' theorem \cite{takens}, it is reasonable to expect $k \leq 2\dim M + 1$ for a generic class of metrics. We hope to prove this in the sequel.

	\bibliography{references}{} 

\def\cprime{$'$} \def\cprime{$'$}
\begin{thebibliography}{BCDS94b}

\bibitem[ADSK16]{ADK16}
Artur Avila, Jacopo De~Simoi, and Vadim Kaloshin.
\newblock An integrable deformation of an ellipse of small eccentricity is an
  ellipse.
\newblock {\em Annals of Mathematics}, pages 527--558, 2016.

\bibitem[BCDS94a]{BCDS94}
Peter Buser, John Conway, Peter Doyle, and Klaus-Dieter Semmler.
\newblock Some planar isospectral domains.
\newblock {\em International Mathematics Research Notices}, 1994(9):391--400,
  1994.

\bibitem[BCDS94b]{Conway}
Peter Buser, John Conway, Peter Doyle, and Klaus-Dieter Semmler.
\newblock Some planar isospectral domains.
\newblock {\em Internat. Math. Res. Notices}, (9):391ff., approx. 9 pp.\, 1994.

\bibitem[BG10]{butler}
Steve Butler and Jason Grout.
\newblock A construction of cospectral graphs for the normalized laplacian.
\newblock {\em arXiv preprint arXiv:1008.3646}, 2010.

\bibitem[Bla00]{blair2000spaces}
David~E Blair.
\newblock Spaces of metrics and curvature functionals.
\newblock In {\em Handbook of differential geometry}, volume~1, pages 153--185.
  Elsevier, 2000.

\bibitem[CB23]{coutinho23}
Gabriel Coutinho and Pedro~Ferreira Baptista.
\newblock Quantum walks in the normalized laplacian.
\newblock {\em Linear and Multilinear Algebra}, pages 1--11, 2023.

\bibitem[DG75]{DG}
J.~J. Duistermaat and V.~W. Guillemin.
\newblock The spectrum of positive elliptic operators and periodic
  bicharacteristics.
\newblock {\em Invent. Math.}, 29(1):39--79, 1975.

\bibitem[DSKW17]{DKW16}
Jacopo De~Simoi, Vadim Kaloshin, and Qiaoling Wei.
\newblock Dynamical spectral rigidity among z2-symmetric strictly convex
  domains close to a circle.
\newblock {\em Ann. of Math.(2)}, 186(1):277--314, 2017.

\bibitem[Dui96]{DuistermaatFIOs}
J.~J. Duistermaat.
\newblock {\em Fourier Integral Operators}.
\newblock Birkh\"auser Boston, 1996.

\bibitem[Eva10]{EvansPDE}
Lawrence~C. Evans.
\newblock {\em Partial differential equations}, volume~19 of {\em Graduate
  Studies in Mathematics}.
\newblock American Mathematical Society, Providence, RI, second edition, 2010.

\bibitem[God93]{godsil93}
Chris Godsil.
\newblock {\em Algebraic combinatorics}, volume~6.
\newblock CRC Press, 1993.

\bibitem[G{\"u}n88]{gunther}
Paul G{\"u}nther.
\newblock {\em Huygens' principle and hyperbolic equations}.
\newblock Boston, MA etc.: Academic Press, Inc., 1988.

\bibitem[GWW92]{GWW92}
Carolyn Gordon, David Webb, and Scott Wolpert.
\newblock Isospectral plane domains and surfaces via riemannian orbifolds.
\newblock {\em Inventiones mathematicae}, 110(1):1--22, 1992.

\bibitem[H{\"o}r71]{HormanderPaper}
Lars H{\"o}rmander.
\newblock {Fourier integral operators. I}.
\newblock {\em Acta Math.}, 127:79--183, 1971.

\bibitem[H{\"o}r85]{HIII}
L.~H{\"o}rmander.
\newblock {\em The analysis of linear partial differential operators. {III}},
  volume 274 of {\em Grundlehren der Mathematischen Wissenschaften [Fundamental
  Principles of Mathematical Sciences]}.
\newblock Springer-Verlag, Berlin, 1985.
\newblock Pseudodifferential operators.

\bibitem[H{\"o}r90]{HI}
Lars H{\"o}rmander.
\newblock {\em The analysis of linear partial differential operators. {I}}.
\newblock Springer-Verlag, 2nd edition, 1990.

\bibitem[H{\"o}r94]{HormanderIV}
Lars H{\"o}rmander.
\newblock {\em The Analysis of Linear Partial Differential Operators IV}.
\newblock Springer-Verlag Berlin Heidelberg, 1994.

\bibitem[HPMS67]{heatsecondterm}
Jr. H.~P.~McKean and I.~M. Singer.
\newblock {Curvature and the eigenvalues of the Laplacian}.
\newblock {\em Journal of Differential Geometry}, 1(1-2):43 -- 69, 1967.

\bibitem[Hsu08]{hsu}
Elton~P Hsu.
\newblock A brief introduction to brownian motion on a riemannian manifold.
\newblock {\em lecture notes}, 2008.

\bibitem[HZ12]{HZ12}
Hamid Hezari and Steve Zelditch.
\newblock {$C^\infty$} spectral rigidity of the ellipse.
\newblock {\em Anal. PDE}, 5(5):1105--1132, 2012.

\bibitem[HZ22]{HZ22}
Hamid Hezari and Steve Zelditch.
\newblock {One can hear the shape of ellipses of small eccentricity}.
\newblock {\em Annals of Mathematics}, 196(3):1083 -- 1134, 2022.

\bibitem[Ivr80]{Ivrii}
V.~Ja. Ivri\u{\i}.
\newblock The second term of the spectral asymptotics for a
  {L}aplace-{B}eltrami operator on manifolds with boundary.
\newblock {\em Funktsional. Anal. i Prilozhen.}, 14(2):25--34, 1980.

\bibitem[Kac66]{Kac66}
Mark Kac.
\newblock Can one hear the shape of a drum?
\newblock {\em Amer. Math. Monthly}, 73(4, part II):1--23, 1966.

\bibitem[KS18]{KS18}
Vadim Kaloshin and Alfonso Sorrentino.
\newblock On the local birkhoff conjecture for convex billiards.
\newblock {\em Annals of Mathematics}, 188(1):315--380, 2018.

\bibitem[Mil64]{milnor}
John Milnor.
\newblock Eigenvalues of the laplace operator on certain manifolds.
\newblock {\em Proceedings of the National Academy of Sciences},
  51(4):542--542, 1964.

\bibitem[Pis16]{MSE}
Emilio Pisanty.
\newblock Can you hear the shape of a drum by choosing where to drum it?
\newblock
  \url{https://mathoverflow.net/questions/227707/can-you-hear-the-shape-of-a-drum-by-choosing-where-to-drum-it},
  2016.

\bibitem[PT03]{PT03}
G~Popov and P~Topalov.
\newblock Liouville billiard tables and an inverse spectral result.
\newblock {\em Ergodic theory and Dynamical systems}, 23(1):225--248, 2003.

\bibitem[PT12]{PT12}
Georgi Popov and Peter Topalov.
\newblock Invariants of isospectral deformations and spectral rigidity.
\newblock {\em Communications in Partial Differential Equations},
  37(3):369--446, 2012.

\bibitem[PT16]{PT16}
G~Popov and P~Topalov.
\newblock From kam tori to isospectral invariants and spectral rigidity of
  billiard tables.
\newblock {\em arXiv preprint arXiv:1602.03155}, 2016.

\bibitem[Sch73]{schwenk73}
Allen~J Schwenk.
\newblock Almost all trees are cospectral.
\newblock {\em New directions in the theory of graphs}, pages 275--307, 1973.

\bibitem[Sog14]{Hang}
C.~D. Sogge.
\newblock {\em Hangzhou lectures on eigenfunctions of the {L}aplacian}, volume
  188 of {\em Annals of Mathematics Studies}.
\newblock Princeton University Press, Princeton, NJ, 2014.

\bibitem[Sog17]{SFIO}
C.~D. Sogge.
\newblock {\em Fourier integrals in classical analysis}, volume 210 of {\em
  Cambridge Tracts in Mathematics}.
\newblock Cambridge University Press, Cambridge, 2nd edition, 2017.

\bibitem[SZ02]{SZDuke}
C.~D. Sogge and S.~Zelditch.
\newblock Riemannian manifolds with maximal eigenfunction growth.
\newblock {\em Duke Math. J.}, 114(3):387--437, 2002.

\bibitem[Tak81]{takens}
Floris Takens.
\newblock Detecting strange attractors in turbulence.
\newblock In David Rand and Lai-Sang Young, editors, {\em Dynamical Systems and
  Turbulence, Warwick 1980}, pages 366--381, Berlin, Heidelberg, 1981. Springer
  Berlin Heidelberg.

\bibitem[Uhl76]{Uhlenbeck}
Karen Uhlenbeck.
\newblock Generic properties of eigenfunctions.
\newblock {\em American Journal of Mathematics}, 98(4):1059--1078, 1976.

\bibitem[Vig21]{Vi21}
Amir Vig.
\newblock Robin spectral rigidity of the ellipse.
\newblock {\em The Journal of Geometric Analysis}, 31(3):2238--2295, 2021.

\bibitem[WX23]{2term}
Emmett~L Wyman and Yakun Xi.
\newblock A two term kuznecov sum formula.
\newblock {\em Communications in Mathematical Physcis}, 2023.
\newblock In press.

\bibitem[Zel92]{z92}
Steven Zelditch.
\newblock Kuznecov sum formulae and szeg{\"o} limit formulae on manifolds.
\newblock {\em Communications in partial differential equations},
  17(1-2):221--260, 1992.

\bibitem[Zel97]{zoll}
Steve Zelditch.
\newblock Fine structure of zoll spectra.
\newblock {\em Journal of Functional Analysis}, 143(2):415--460, 1997.

\bibitem[Zel04]{Z04}
Steve Zelditch.
\newblock Inverse spectral problem for analytic domains {I}: Balian-bloch trace
  formula.
\newblock {\em Communications in mathematical physics}, 248(2):357--407, 2004.

\bibitem[Zel09]{Z09}
Steve Zelditch.
\newblock Inverse spectral problem for analytic domains, {II}: {$\mathbb
  Z^2$}-symmetric domains.
\newblock {\em Annals of mathematics}, pages 205--269, 2009.

\end{thebibliography}
	\bibliographystyle{alpha}
	
\end{document}